\documentclass[a4paper]{article}

\usepackage{amsmath,amsthm,amssymb} 
\usepackage{hyperref}
\usepackage[applemac]{inputenc}

\usepackage[T1]{fontenc}

\usepackage[a4paper, left=3cm, right=3cm, top=4cm, bottom=4cm]{geometry}

\DeclareMathOperator{\Id}{Id}

\DeclareMathOperator{\tr}{Tr}

\DeclareMathOperator{\pr}{pr}

\DeclareMathOperator{\SO}{SO}

\DeclareMathOperator{\GL}{GL}

\DeclareMathOperator{\Rm}{Rm}
\DeclareMathOperator{\RR}{\mathbf{R}}
\DeclareMathOperator{\Ric}{Ric}
\DeclareMathOperator{\vol}{vol}

\newcommand{\R}{\mathbb R}

\newcommand{\Z}{\mathbb Z}
\newcommand{\N}{\mathbb N}

\newcommand{\diff}{\text{\rm d}}
\newcommand{\dd}{\text{\rm d}}
\newcommand{\del}{\partial}

\newcommand{\dvol}{\mathrm{dvol}}

\theoremstyle{plain}
	\newtheorem{theorem}{Theorem}
	\newtheorem{proposition}[theorem]{Proposition}
	\newtheorem{lemma}[theorem]{Lemma}
	\newtheorem{corollary}[theorem]{Corollary}
	\newtheorem{conjecture}[theorem]{Conjecture}

\theoremstyle{definition}
	\newtheorem{definition}[theorem]{Definition}
	\newtheorem{remark}[theorem]{Remark}

\theoremstyle{plain}
	\newtheorem*{theorem*}{Theorem}
	\newtheorem*{proposition*}{Proposition}
	\newtheorem*{lemma*}{Lemma}
	\newtheorem*{corollary*}{Corollary}
	\newtheorem*{conjecture*}{Conjecture}
\theoremstyle{definition}
	\newtheorem*{definition*}{Definition}
	\newtheorem*{remark*}{Remark}
	\newtheorem*{remarks*}{Remarks}
	
\numberwithin{theorem}{section}

\usepackage{authblk}
\title{Hypersymplectic 4-manifolds, the $G_2$-Laplacian flow and extension assuming bounded scalar curvature}
\author{Joel Fine and Chengjian Yao\\ Universit\'e libre de Bruxelles}
\date{}                     

\begin{document}

\maketitle

\begin{abstract}
A hypersymplectic structure on a 4-manifold $X$ is a triple $\underline{\omega}$ of symplectic forms which at every point span a maximal positive-definite subspace of $\Lambda^2$ for the wedge product. This article is motivated by a conjecture of Donaldson: when $X$ is compact $\underline{\omega}$ can be deformed through cohomologous hypersymplectic structures to a hyperk\"ahler triple. We approach this via a link with $G_2$-geometry. A hypersymplectic structure $\underline\omega$ on a compact manifold $X$ defines a natural $G_2$-structure $\phi$ on $X \times \mathbb{T}^3$ which has vanishing torsion precisely when $\underline{\omega}$ is a hyperk\"ahler triple. We study the $G_2$-Laplacian flow starting from $\phi$, which we interpret as a flow of hypersymplectic structures. Our main result is that the flow extends as long as the scalar curvature of the corresponding $G_2$-structure remains bounded.  An application of our result is a lower bound for the maximal existence time of the flow, in terms of weak bounds on the initial data (and with no assumption that scalar curvature is bounded along the flow). 
\end{abstract}

\section{Introduction}\label{intro}

Let $X$ be an oriented 4-manifold. A \emph{hypersymplectic structure} on $X$ is a triple $\underline{\omega} = (\omega_1, \omega_2, \omega_3)$ of closed 2-forms which at every point span a maximal positive-definite subspace of $\Lambda^2$ for the wedge product. In particular, each $\omega_i$ is symplectic. Hypersymplectic structures were introduced by Donaldson \cite{D09} and they play an important role in the programme he has introduced to study the adiabatic limit of $G_2$-manifolds \cite{Donaldson}. A special case of this definition appears in earlier work of Geiges \cite{Geiges} which studies both pairs and triples of symplectic forms with $\omega_i \wedge \omega_j =0$ for $i \neq j$. Hypersymplectic structures have also appeared in the work of Madsen \cite{madsen} and Madsen--Swann \cite{madsen-swann} on reductions of metrics with special holonomy. 

The simplest example of a hypersymplectic structure is the triple of K\"ahler forms of a hyperk\"ahler metric. The main motivation for this article is a conjecture of Donaldson that on a compact 4-manifold and up to isotopy this is the \emph{only} example.

\begin{conjecture}[Donaldson \cite{D09}]\label{skd-conjecture}
Let $(X, \underline{\omega})$ be a compact 4-manifold with a hypersymplectic structure. Suppose moreover that $\int \omega_i\wedge \omega_j = \delta_{ij}$. Then there exists a deformation of $\underline{\omega}$ through cohomologous hypersymplectic structures to a hyperk\"ahler triple.
\end{conjecture}

(Note that one can always apply a constant linear transformation to a given hypersymplectic structure to ensure $\int \omega_i \wedge \omega_j = \delta_{ij}$.)

Donaldson's conjecture can be seen as a special case of an important folklore conjecture in 4-dimensional symplectic geometry: if $(X,\omega)$ is a compact symplectic 4-manifold with $c_1=0$ and $b_+=3$, then there is a compatible complex structure on $X$ making it into a 
hyperk\"ahler manifold. Now given a hypersymplectic structure $\underline{\omega}$ on $X$, we have $c_1(X,\omega_1)=0$ and $b_+=3$. To see this, consider the conformal structure on $X$ for which $\Lambda^+ = \left\langle \omega_i \right\rangle$; this determines an almost complex structure  compatible with $\omega_1$ and whose canonical bundle is isomorphic to the sub-bundle of $\Lambda^+$ which is orthogonal to $\omega_1$. The form $\omega_2$, say, projects to give a nowhere vanishing section of the canonical bundle, showing that $c_1=0$. Since the $\omega_i$ are independent (at every point even), it follows that $b_+ \geq 3$ and a theorem of Bauer then shows that in fact $b_+ = 3$ \cite[Corollary~1.2]{Bauer}. Assuming Conjecture~\ref{skd-conjecture}, Moser's theorem then proves the folklore conjecture for $(X,\omega_1)$. 


This article investigates a geometric flow which is designed to deform a given hypersymplectic structure towards a hyperk\"ahler one. The flow in question is a dimensional reduction of the $G_2$-Laplacian flow, which we now describe. We begin with a rapid overview of $G_2$-structures, to fix notation. Let $M$ be a 7-manifold and $\phi$ a 3-form on $M$. There is a symmetric bilinear form $B_\phi$ on $TM$ with values in $\Lambda^7$ given by
\[
B_\phi(u,v) = \frac{1}{6} \iota_u\phi \wedge \iota_v \phi \wedge \phi
\]
When $B_\phi$ is definite $\phi$ is called a \emph{$G_2$-structure}. In this case there is a unique Riemannian metric $g_\phi$ and orientation with the property that $g_\phi \otimes \dvol_{g_\phi} = B_\phi$.  If $\nabla \phi = 0$ (where $\nabla$ is the Levi-Civita connection of $g_\phi$) then $g_\phi$ has holonomy contained in $G_2$. When this happens we say $\phi$ defines a \emph{$G_2$-metric}.

A $G_2$-structure is called \emph{closed} when $\diff\phi = 0$. A central question is to decide, given a closed $G_2$-structure $\phi$, whether or not there exists a genuine $G_2$-metric in the cohomology class $[\phi]$. To this end, Hitchin \cite{Hitchin} studied the total volume functional $\mathcal{V}(\phi) = \int \dvol_{g_\phi}$ restricted to $[\phi]$. He proved that the critical points are precisely those $\phi$ defining $G_2$-metrics. A natural approach is then to consider the gradient flow of $\mathcal{V}$. This gives the evolution equation
\begin{equation}
\del_t \phi = \Delta_{\phi} \phi
\label{G2-flow}
\end{equation}
where $\Delta_{\phi}$ is the Hodge Laplacian of $g_\phi$. This flow, called the \emph{$G_2$-Laplacian flow}, was also independently introduced by Bryant in \cite{Bryant}. Bryant--Xu proved short-time existence and uniqueness of the flow on compact manifolds in \cite{BX}. 

We now relate this to hypersymplectic structrures (following Donaldson's article \cite{Donaldson}). Given a hypersymplectic structure $\underline{\omega}$ on $X$ we consider the following $3$-form $\phi$ on the 7-manifold $M = X \times \mathbb{T}^3$. We use angular coordinates $t^1, t^2, t^3$ on $\mathbb{T}^3 = S^1 \times S^1 \times S^1$ and define
\begin{equation}
\phi = \diff t^{123} - \diff t^1 \wedge \omega_1 - \diff t^2 \wedge \omega_2 - \diff t^3 \wedge \omega_3
\label{phi-from-omega}
\end{equation}
One checks that $\phi$ is a closed $G_2$-structure precisely because $\underline{\omega}$ is hypersymplectic (see Lemma~\ref{G2-from-hypersymplectic} below). The $G_2$-Laplacian flow for $\phi$ descends to a flow of hypersymplectic structures on $X$ which we describe next. 

As we have already mentioned, a hypersymplectic structure $\underline{\omega}$ determines a conformal structure on $X$, for which $\Lambda^+ = \left\langle \omega_i \right\rangle$. Given any choice of volume form $\mu$, we obtain a Riemannian metric in this conformal class and using this we can define the $3\times 3$ symmetric matrix valued function
\[
Q_{ij} 
=
\frac{\omega_i \wedge \omega_j}{2\mu}
\]
There is a unique volume form $\mu$ for which $\det Q=1$ and, with this convention, $\underline{\omega}$ defines a Riemannian metric $g_{\underline{\omega}}$ on $X$. The triple $\underline{\omega}$ is a hyperk\"ahler triple precisely when $Q = \Id$ and when $Q$ is constant $g_{\underline{\omega}}$ is hyperk\"ahler with $\underline{\omega}$ a constant  linear combination of a hyperk\"ahler triple. The matrix $Q$ and metric $g_{\underline{\omega}}$ are related to the 7-dimensional metric $g_\phi$ via
\[
g_\phi = g_{\underline{\omega}} \oplus Q_{ij} \diff t^i \diff t^j
\]
with respect to the natural splitting $TM \cong TX \oplus T(\mathbb{T}^3)$. (These claims are all proved in \S\ref{g-phi-g-omega}.)

With this in hand we can now describe the $G_2$-Laplacian flow on $X \times \mathbb{T}^3$ in terms of $\underline{\omega}$. If the flow begins with initial condition $\phi$ of the form \eqref{phi-from-omega}, then $\phi(t)$ is of the same form for all~$t$ (Lemma~\ref{G2-flow-descends} below) and corresponds to a flow of hypersymplectic structures which evolve according to the equation
\begin{equation}
\del_t \underline{\omega} = \diff \left( Q\, \diff^* ( Q^{-1}\underline{\omega} )  \right)
\label{hypersymplectic-flow}
\end{equation}
(The notation used here is that if $S_{ij}$ is a $3\times 3$-matrix and $\underline{\alpha}$ is a triple of forms, then $S \underline{\alpha}$ denotes the triple $(S\underline\alpha)_i = S_{ij} \alpha_j$.) We call \eqref{hypersymplectic-flow} the \emph{hypersymplectic flow} and it is the focus of this article. 

To put our main result in context, we first recall what is known about the $G_2$-Laplacian flow in general. The \emph{torsion} of a closed $G_2$-structure $\phi$ is the 2-form $T = -\frac{1}{2} \diff^* \phi$, which vanishes precisely when $\phi$ determines a genuine $G_2$-metric. Lotay--Wei proved an extension theorem \cite[Thmeroem~1.6]{Lotay--Wei}, based on Bernstein--Bando--Shi estimates involving both torsion and curvature, more precisely using the quantity 
\[
\Lambda(\phi) = \sup_M \left(  | \Rm(g_\phi)^2 | + | \nabla T|^2 \right)^{1/2}
\]
Lotay--Wei's extension theorem is:
\begin{theorem}[Lotay--Wei \cite{Lotay--Wei}]
Let $\phi(t)$ be a solution of the $G_2$-Laplacian flow \eqref{G2-flow} on a compact 7-manifold and time interval $t \in [0,s)$ with $s < \infty$. If $|\Delta_\phi\phi|$ is bounded on $[0,s)$ then the flow extends beyond $t=s$, to the time interval $[0,s+\epsilon)$ for some $\epsilon >0$. 
\end{theorem}
This is, of course, the $G_2$ analogue of Sesum's theorem that Ricci flow exists as long as the Ricci curvature is bounded \cite{Sesum}. The main result of this article is that for the hypersymplectic flow, one can replace $|\Delta_\phi \phi|$ by the a~priori weaker quantity $|T|$:
\begin{theorem}\label{main-theorem}
Let $\underline{\omega}(t)$ be a solution of the hypersymplectic flow \eqref{hypersymplectic-flow} on a compact 4-manifold $X$ and time interval $t \in [0, s)$ with $s < \infty$. If $| T |$ is bounded on $[0,s)$ then the flow extends beyond $t=s$ to the time interval $[0,s+\epsilon)$ for some $\epsilon >0$. 
\end{theorem}
A calculation of Bryant \cite{Bryant} says that the scalar curvature of a closed $G_2$-structure $\phi$ is $R(g_\phi) = - |T|^2$. Our result says that given a hypersymplectic structure $\underline{\omega}(0)$, the $G_2$-Laplacian flow on $X \times \mathbb{T}^3$ starting from $\phi(0)$ given by \eqref{phi-from-omega} can be continued for as long as the \emph{scalar} curvature of the metric $g_{\phi(t)}$ remains bounded. It is interesting to note that this is much stronger than what is currently known for the Ricci flow (cf.~the works \cite{EMT, W1, W2, CW, Z} for the study of Ricci flow under bounded scalar curvature). 

A consequence of Theorem~\ref{main-theorem} is a lower bound for the existence time of a hypersymplectic flow, purely in terms of a $C^1$ bound on $Q$ at $t=0$. We do not assume here that the flow has bounded torsion. We give a loose statement of the result here, see Theorem~\ref{lower-bound-existence-time} for the optimal statement.

\begin{theorem}\label{existence-time}
Let $K>0$. There exists a constant $\epsilon>0$ depending only on $K$, such that whenever $\underline{\omega}(0)$ is a hypersymplectic structure with $\| Q\|_{C^1} \leq K$ then the hypersymplectic flow $\underline{\omega}(t)$ starting at $\underline{\omega}(0)$ exists for all $t \in [0,\epsilon]$. 
\end{theorem}

We now give a brief outline of the proof of Theorem~\ref{main-theorem}, the details of which occupy the rest of the paper. A key point is that when $\underline{\omega}(t)$ solves the hypersymplectic flow, the corresponding metric $g(t)$ and positive definite matrix $Q(t)$ solve a version of the coupled harmonic--Ricci flow (introduced by Reto~Buzano (n\'e M\"uller) in \cite{Muller}) with additional ``forcing'' terms involving $T$. The proof exploits this by combining ideas from Ricci flow and the harmonic map flow. 

We assume for a contradiction that the flow does not extend. We consider a parabolic rescaling of $\underline{\omega}(t)$ as $t$ approaches the maximal time and, using Lotay--Wei's estimates on $\Lambda(\phi)$, take a limit. To do this, we show that the bound on $|T|$ implies the flows are noncollapsed.  We use here a recent result of Gao Chen \cite[Theorem~1.1]{Chen} which generalises Perelman's $\kappa$-noncollapsing theorem for the Ricci flow \cite{perelman} to a metric flow $g(t)$ for which $\del_t g$ is a bounded distance from $-2\Ric$. To prove the hypotheses of Chen's theorem are satisfied, we use the maximum principle, in the style of the harmonic map flow, to control $\diff Q$, and show that bounds on torsion imply that the flow is uniformly noncollapsed, in finite time. 

The next step, and the real crux of the argument, is to show that the limit is an asymptotically locally Euclidean (ALE) gravitational instanton. This sort of conclusion is currently out of reach in Ricci flow, and this explains the big difference between what is known there and what we are able to prove for the hypersymplectic flow. There are two separate things to show: proving that the limit is hyperk\"ahler and proving that it has finite energy, i.e.~finite $L^2$-norm of curvature. (The fact that the limit is ALE follows from the uniform noncollapsing, which ensures Euclidean volume growth in the limit.) 

To achieve the first part, we first bound the $C^2$-norm of $\underline\omega$ in terms of bound of $\dd Q$ and Riemannian curvature. 
This enables us to take a limit of the rescaled hypersymplectic structures, giving a hyperk\"ahler triple on the limit. The proof that the limit has finite energy is quite delicate. We show that the bound on $|T|$ gives a bound on the energy along the hypersymplectic flow (in finite time). By scaling invariance, this translates to an energy bound in the limit. Our argument here is inspired by a result of Miles Simon \cite{Simon} which shows that a bound on scalar curvature of a compact 4-dimensional Ricci flow implies a bound on the energy in finite time. Simon's proof involves integral estimates of two different functions, the ``bad'' term of one being cancelled by the ``good'' term of the other. Here the calculations are more complicated, with additional bad terms appearing when Simon's two quantities are considered. We are able to complete the proof by a judicious choice of two additional functions which generate the required good terms. 

To complete the proof, we invoke Kronheimer's classification of ALE gravitaional instantons \cite{Kron} from which it follows that the limit contains a 2-sphere which is holomorphic for one of its hyperk\"ahler complex structures. From this, and the fact that the hyperk\"ahler triple is a scaling limit of the hypersymplectic structures, we find a contradiction using a topological argument. 
 
This article is organised as follows. In \S\ref{short-time-existence} we show short-time existence and uniqueness for the hypersymplectic flow \eqref{hypersymplectic-flow}. This follows from the analogous result for the $G_2$-Laplacian flow, together with a calculation of the corresponding flow on the 4-manifold. In \S\ref{identities} we give a series of identities relating various geometric quantities on $(X, g_{\underline{\omega}})$ and $(X\times \mathbb{T}^3, g_\phi)$. In \S\ref{evolution-equations} we derive the necessary evolution equations. Some can be deduced directly from the known equations for the $G_2$-Laplacian flow, others have purely 4-dimensional derivations. We also prove the required maximum principles here. \S\ref{L2-curvature} is the technical heart of the paper, proving that the $L^2$-norm of the curvature of  $g_{\underline{\omega}}$ is also bounded in finite time. \S\ref{proof-of-main-result}  then assembles all the parts of the proof of Theorems~\ref{main-theorem} and~\ref{existence-time}.

\subsection*{Acknowledgements}

It is a pleasure to acknowledge important discussions with Gao Chen, Xiuxiong Chen, Simon Donaldson and Jason Lotay  during the course of this work. Both authors were supported by FNRS grant MIS.F.4522.15. JF was also supported by ERC consolidator grant 646649 ``SymplecticEinstein''. Part of this work was carried out in Spring 2016 whilst CY was a Viterbi Fellow at MSRI and JF was a visitor there. We are grateful for the support of MSRI, financed in part by the NSF grant DMS-1440140. 

\section{Preliminary definitions and short time existence}
\label{short-time-existence}

The main result of this section is:

\begin{proposition}\label{short-time-existence-theorem}
Let $\underline{\omega}$ be a hypersymplectic structure on a compact 4-manifold $X$. Then there exists a unique short time solution $\underline{\omega}(t)$ to the hypersymplectic flow \eqref{hypersymplectic-flow} starting at $\underline{\omega}$.
\end{proposition}

This will follow easily from the analogous result of Bryant--Xu for the $G_2$-Laplacian flow. Along the way we make explicit the relationship between the natural metric $g_{\underline{\omega}}$ on $X$ induced by a hypersymplectic structure $\underline{\omega}$ and the metric $g_\phi$ induced on $X \times \mathbb{T}^3$ induced by the $G_2$-structure $\phi$ given by \eqref{phi-from-omega}.

\subsection{Relating $g_{\underline{\omega}}$ and $g_\phi$}\label{g-phi-g-omega}

\begin{lemma}\label{G2-from-hypersymplectic}
Given a triple $\underline{\omega}$ of  2-forms on $X$, let $\phi$ be the 3-form on $X \times \mathbb{T}^3$ defined by \eqref{phi-from-omega}. Then $\phi$ is a $G_2$-structure if and only if $\left\langle \omega_i \right\rangle \subset \Lambda^2(T^*X)$ is a maximal definite subspace for the wedge product.
\end{lemma}

\begin{proof}
Let $\underline{a} = (a_1,a_2,a_3) \in \R^3$ and write $u = a_1 \del_{1} + a_2 \del_{2} + a_3 \del_{3}$, where $\del_{i}$ are the coordinate vector fields on $\mathbb{T}^3$. Then
\begin{equation}
B_\phi(u,u) 
=
\frac{1}{6}\iota_u \phi \wedge \iota_u \phi \wedge \phi
=
\frac{1}{2}\diff t^{123} \wedge \sum_{i,j=1}^3a_ia_j\omega_i \wedge\omega_j
\label{Bphi-vertical}
\end{equation}
If $\phi$ is a $G_2$-form then $B_\phi(u,u) \neq 0$ whenever $\underline{a} \neq 0$ and hence $\left\langle \omega_i \right\rangle$ is a maximal definite subspace for the wedge product. 

Conversely, if $\left\langle \omega_i \right\rangle$ is a maximal definite space for the wedge product, $B_\phi$ is definite on vectors tangent to the $\mathbb{T}^3$ factor.  Meanwhile, if $u$ is tangent to $\mathbb{T}^3$, and $v, w$ are tangent to $X$, then one checks that $B_\phi(u,v) = 0$ whilst $B_\phi(v,w) = \diff t^{123} \wedge \beta_{\underline{\omega}}(v,w)$ where 
\begin{equation}
\beta_{\underline{\omega}}(v,w)= 
\frac{1}{6} \sum_{i,j,k=1}^3
\epsilon^{ijk}\iota_v\omega_i\wedge\iota_w\omega_j\wedge\omega_k
\label{Bphi-horizontal}
\end{equation}
(Here $\epsilon^{ijk}$ is the sign of the permutation $(i,j,k)$.) We must check that $\beta_{\underline{\omega}}$ is definite. Let $\theta_i$ be a basis for $\left\langle \omega_i \right\rangle$ which diagonalises the wedge product, i.e., such that $\theta_i \wedge \theta_j = 2\delta_{ij} \mu$ for some volume form $\mu$. We can write $\omega_i = A_{ij} \theta_j$ in this basis. Then
\begin{equation}
\beta_{\underline{\omega}}(v,w)
=
\frac{1}{6} \sum_{i,j,k,p,q,r} \epsilon^{ijk} A_{ip}A_{jq} A_{kr}
\iota_v \theta_p\wedge \iota_w \theta_q \wedge \theta_r
=
\det(A) \beta_{\underline{\theta}}(v,w)
\label{omega-theta}
\end{equation}
where $\beta_{\underline{\theta}}$ is given by the same formula as $\beta_{\underline{\omega}}$ but with $\theta_i$ in place of $\omega_i$. Now the $\theta_i$ are standard; one can find a coframe $e_1, e_2, e_3, e_4$ such that $\theta_1 = e_1 \wedge e_2 + e_3 \wedge e_4$ etc. A direct calculation gives that $\beta_{\underline\theta}(v,v) = |v|^2 e_{1234}$ where $|v|^2$ is computed in the metric which makes the $e_i$ an orthonormal coframe. It follows that $\beta_{\underline{\omega}}$ is also definite, which completes the proof.
\end{proof}

\begin{definition}~
\begin{enumerate}
	\item
	Given a hypersymplectic structure $\underline{\omega}$, the above proof shows that $\beta_{\underline{\omega}}$ defines a $\Lambda^4$-valued definite bilinear form on $TX$. This defines a conformal structure on $X$ and, just as for $G_2$-structures, there is a unique volume form $\mu$ on $X$ for which the resulting metric $g_{\underline{\omega}}$ in this conformal class satisfies $g_{\underline{\omega}} \otimes \mu = \beta_{\underline{\omega}}$. We call $g_{\underline{\omega}}$ and $\mu$ the \emph{metric and volume form induced by $\underline{\omega}$}.
	\item
	Write $Q \colon X \to S^2\R^3$ for the symmetric $3\times 3$ matrix-valued function defined by
	\[
	Q_{ij} 
	= \frac{\omega_i\wedge \omega_j}{2\mu} 
	\]
	(where $\mu$ is the volume form induced by $\underline{\omega}$). $Q$ will play a central role throughout this article.
\end{enumerate}
\end{definition}

\begin{lemma}\label{sd-detQ}
The 2-forms $\omega_i$'s are self-dual with respect to the metric $g_{\underline{\omega}}$ and $\det Q =1$.
\end{lemma}
\begin{proof}
This is essentially contained in the proof of Lemma~\ref{G2-from-hypersymplectic}. Let $\theta_i$ be a basis for $\left\langle \omega_i \right\rangle$ with $\theta_i \wedge \theta_j = 2\delta_{ij} \mu$, where $\mu$ is the volume form of $g_{\underline{\omega}}$. The $\theta_i$ satisfy the pointwise conditions of a hyperk\"ahler triple (although they certainly need not be closed). We see from \eqref{omega-theta} and the discussion immediately afterwards that $g_{\underline{\omega}}$ is the resulting metric defined in a standard fashion from this ``quaternionic'' triple $\theta_i$.  It follows firstly that $\Lambda^+ = \left\langle \theta_i \right\rangle = \left\langle \omega_i \right\rangle$. Moreover, if we write $\omega_i = A_{ij} \theta_j$, then \eqref{omega-theta} implies that $\det A = 1$. Since $Q_{ij} = A_{ip}A_{pj}$ we have $\det Q = 1$ as well. 
\end{proof}

\begin{lemma}
Let $\underline{\omega}$ be a hypersymplectic structure on $X$ and $\phi$ the corresonding $G_2$-structure on $X \times \mathbb{T}^3$ given by \eqref{phi-from-omega}. Then, with respect to the natural splitting $T(X\times \mathbb{T}^3) = TX \oplus T\mathbb{T}^3$, we have
\begin{align*}
\dvol_{g_\phi} &= \mu \wedge \diff t^{123}\\
g_\phi &= g_{\underline{\omega}} \oplus Q_{ij} \diff t^i \diff t^j
\end{align*}
\end{lemma}
\begin{proof}
From the proof of Lemma~\ref{G2-from-hypersymplectic}, in particular \eqref{Bphi-horizontal} and the discussion leading up to \eqref{Bphi-vertical}, we have
\[
B_\phi 
= 
\left( g_{\underline{\omega}} + Q_{ij} \diff t^i \diff t^j \right)\otimes \mu \wedge \diff t^{123} 
\]
The result now follows from the fact that $\det Q=1$.
\end{proof}

\subsection{From the $G_2$-Laplacian flow to the hypersymplectic flow}

The goal of this section is to prove that if the $G_2$-Laplacian flow $\phi(t)$ starts from $\phi(0)$ of the form of \eqref{phi-from-omega} then it remains of the same form, thus giving rise to a flow of hypersymplectic structures on $X$. 

\begin{lemma}
Let $\phi(t)$ be the $G_2$-Laplacian flow on $X \times \mathbb{T}^3$ with initial condition $\phi(0)$ given by \eqref{phi-from-omega}, where $\underline{\omega}$ is a hypersymplectic structure. Then $\phi(t)$ is $\mathbb{T}^3$ invariant for as long as it exists.
\end{lemma}
\begin{proof}
This is immediate from the uniqueness part of Bryant--Xu's theorem \cite[Theorem~0.1]{BX}: since the initial data $\phi(0)$ is $\mathbb{T}^3$ invariant, so is the ensuing flow. 
\end{proof}

By $\mathbb{T}^3$-invariance, the 3-form $\phi(t)$ on $X \times \mathbb{T}^3$ necessarily has the shape
\[
\phi(t)
=
A \diff t^{123} 
+ B_1 \wedge \diff t^{23} + B_2 \wedge \diff t^{31} + B_3 \wedge \diff t^{12}
- \diff t^1 \wedge C_1 - \diff t^2 \wedge C_2 - \diff t^3 \wedge C_3
+ 
D
\]
where $A(t) \in \Omega^0(X)$, $B_i(t) \in \Omega^1(X)$, $C_i(t) \in \Omega^2(X)$ and $D(t) \in \Omega^3(X)$ are paths of forms on~$X$. Moreover, since $\diff \phi(t) =0$, it follows that these forms on $X$ are \emph{closed}.
\begin{lemma}
$A(t) = 1$.
\end{lemma}
\begin{proof}
We know that $\diff A(t) = 0$, i.e., that for every $t$, $A(t)$ is constant. Moreover,
\[
\int_{\mathbb{T}^3} \phi(t) = (2\pi)^3 A(t)
\]
Since $\del_t \phi$ is exact, this integral is independent of $t$, and so $A(t) = A(0) = 1$. 
\end{proof}

Next we consider the involution $\vartheta \colon X \times \mathbb{T}^3 \to X \times \mathbb{T}^3$ given by $\vartheta(p, t) \mapsto (p,-t)$ and write $\hat \phi(t) = - \vartheta^* \phi(t)$. 

\begin{lemma}\label{G2-flow-descends}
For all $t$, $\hat{\phi}(t) = \phi(t)$. Hence $B_i(t) = 0 = D(t)$ vanish identically, and $\phi(t)$ remains of the form \eqref{phi-from-omega} for as long as it exists, for a closed triple $\underline{\omega}(t)$ of 2-forms on $X$. 
\end{lemma}
\begin{proof}
We will prove that $\hat{\phi}(t)$ solves the $G_2$-Laplacian flow. Since $\hat{\phi}(0) = \phi(0)$ the Lemma then follows from the uniqueness part of Bryant--Xu's theorem. The minus sign in the definition $\hat{\phi} = - \vartheta^*\phi$, together with the fact that $\vartheta$ is orientation \emph{reversing}, means that $\hat\phi$ is a $G_2$-structure inducing the same orientation as $\phi$. One checks that $g_{\hat{\phi}} = \vartheta^*g_\phi$. Now given any metric $g$ and any diffeomorphism $\psi$, the Hodge--Laplacians of $g$ and $\psi^*g$ are related by $\Delta_{\psi^*g}(\psi^*\alpha) = \psi^*(\Delta_g \alpha)$. It follows that $\hat{\phi}$ solves the $G_2$-Laplacian flow:
\[
\Delta_{\hat{\phi}} \hat{\phi}
=
- \Delta_{\vartheta^*g_{\phi}}\left( \vartheta^*\phi \right)
=
- \vartheta^*\left( \Delta_\phi  \phi \right) 
=
-\vartheta^* \del_t \phi
=
\del_t \hat{\phi}
\qedhere
\]
\end{proof}

It remains to derive the evolution equation \eqref{hypersymplectic-flow} for the hypersymplectic flow. This will complete the proof of Proposition~\ref{short-time-existence-theorem}.

\begin{lemma}
Let $\underline{\omega}$ be a hypersymplectic structure on $X$ and $\phi$ the associated $G_2$-structure on $X \times \mathbb{T}^3$ defined in \eqref{phi-from-omega}. Then
\[
\Delta_\phi \phi
=
-
\sum_{i,p,q} \diff t^i \wedge \diff (Q_{ip}d^*_4(Q^{pq}\omega_q)
\]
where $\diff^*_4$ is the adjoint of $\diff$ on $(X, g_{\underline{\omega}})$ and $Q^{pq}$ is the inverse matrix to $Q_{pq}$. It follows that if $\underline{\omega}(t)$ is the flow of hypersymplectic structures corresponding to the $G_2$-Laplacian flow $\phi(t)$ on $X \times \mathbb{T}^3$, with $\phi(t)$ and $\underline{\omega}(t)$ related by \eqref{phi-from-omega}, then
\[
\del_t \,\underline{\omega} = \diff \left( Q\, \diff^* ( Q^{-1}\underline{\omega} )  \right)
\]
\end{lemma}

\begin{proof}
We write $*_3,*_4$ and $*_7$ for the Hodge stars associated to the metrics $Q_{ij}\diff t^i \otimes \diff t^j$ on $\mathbb{T}^3$, $g_{\underline{\omega}}$ on $X$ and $g_\phi$ on $X \times \mathbb{T}^3$ respectively. We write $\widehat{\diff t}^1 = \diff t^{23}$ etc. Then $*_3 \diff t^i= Q^{ij} \widehat{\diff t}^j$ and $*_3 \widehat{\diff t}^i = Q_{ij} \diff t^j$. From this we have
\[
*_7 \phi
	=
		\mu - \sum *_7\left( \diff t^i \wedge \omega_i \right)
	=
		\mu - \sum_i *_3 \diff t^i \wedge *_4 \omega_i
	=
		\mu - \sum_i \widehat{\diff t}^j \wedge Q^{ij}\omega_i 
\]
Now $\diff^*_7 \phi = - *_7 \diff *_7 \phi$ which is
\begin{equation}
\diff^*_7 \phi
	=
		*_7 \left( 
		\sum_{i,j} \widehat{\diff t}^j \wedge \diff (Q^{ij}\omega_i) 
		\right)
	=
		\sum_{i,j,k} Q_{jk} \diff t^k \wedge *_4 \diff (Q^{ij}\omega_i)
\label{dstar-phi}
\end{equation}
and hence $\Delta_\phi \phi = \diff \diff^*_7 \phi$ is given by
\[
\Delta_\phi \phi
	=
		\sum_{i,j,k} \diff t^k \wedge \diff\left( 
		Q_{jk} *_4 \diff *_4 (Q^{ij} \omega_i) 
		\right)
	=
		- \sum_{i,p,q}
		\diff t^i \wedge \diff (Q_{ip}\diff^*_4 (Q^{pq}\omega_q))
\]
as claimed (where in the last line we have used the fact that $Q$ is symmetric to reorder the indices).
\end{proof}

\section{Identities and inequalities}
\label{identities}

In this section we derive various identities relating the geometries of $g_\phi$ and $g_{\underline{\omega}}$. We will then use these identities to show how geometric quantities can be controlled by $Q$ and its derivatives. At this stage we consider a single hypersymplectic structure, i.e., not evolving with time. 

\subsection{Notation}

It will be convenient to think of $Q$ as a map from $X$ into the space $\mathcal{P}$ of all positive-definite symmetric $3\times 3$-matrices, considered with its non-positively curved symmetric metric. Explicitly, identifying $T_Q \mathcal{P}$ with the space of all symmetric matrices, the Riemannian metric on $\mathcal{P}$ is given by
\[
\left\langle A,B \right\rangle_Q = \tr(Q^{-1}AQ^{-1}B)
\]
We write $\hat{\nabla}$ for the Levi-Civita connection of $\mathcal{P}$. Explicitly, if we identify a pair $A, B$ of symmetric matrices with a pair of vector fields on $\mathcal{P}$, then $\hat{\nabla}_{A}B$ evaluated at $Q$ corresponds to the symmetric matrix $-\frac{1}{2} AQ^{-1}B - \frac{1}{2}B Q^{-1}A$. To see that this is indeed the Levi-Civita connection it suffices to note that it preserves the metric and that it is symmetric in $A$ and $B$, which is equivalent to being tosrion free. 

We use the same notation $\hat{\nabla}$ for the induced connection on tensors over $X$ with values in $Q^*T\mathcal{P}$ obtained from the Levi-Civita connections on $(X, g_{\underline{\omega}})$ and $\mathcal{P}$. So, for example, the Hessian of $Q$ is given by $\hat{\nabla} \diff Q$. We can relate this to the Hessian of the individual components $Q_{ij}$ of $Q$ via
\begin{equation}\label{Hessians-relation}
(\hat{\nabla} \diff Q)_{ij}
=
\nabla (\diff Q_{ij}) 
- \frac{1}{2}Q^{pq} \left( 
\diff Q_{ip} \otimes \diff Q_{qj} + \diff Q_{qj}\otimes \diff Q_{ip} 
\right)
\end{equation}
This follows from the above formula for the Levi-Civita connection of $\mathcal{P}$.

We write $\hat{\Delta}Q =  \tr_{g_{\underline{\omega}}} \hat{\nabla}\diff Q$ for the Laplacian of the map $Q \colon X \to \mathcal{P}$ of Riemannian manifolds, where we use the metric $g_{\underline{\omega}}$ on $X$ and the symmetric metric on $\mathcal{P}$. This is the Laplacian which appears in the harmonic map equation (sometimes called the ``tension field''). Again, the ``hat'' notation is to differentiate this from the Laplacian of $Q$ thought of simply as a map to the affine space of all matrices, i.e., the ordinary Laplacian taken component-by-component, $(\Delta Q)_{ij} = \Delta (Q_{ij})$. The two Laplacians are related by 
\[
(\hat{\Delta} Q )_{ij}
= 
\Delta(Q_{ij}) - Q^{pq}\left\langle \diff Q_{ip}, \diff Q_{qj} \right\rangle. 
\]

We will frequently move back and forth between the 4-dimensional hypersymplectic structure and the corresponding 7-dimensional $G_2$-structure. We will use bold symbols for those quantities associated to the 7-dimensional Riemannian manifold $(X \times \mathbb{T}^3, g_{\phi})$ and normal symbols for those associated to the 4-dimensional Riemannian manifold $(X, g_{\underline{\omega}})$. So, for example, $\mathbf{R}$ denotes the scalar curvature of~$g_{\phi}$ whilst $R$ denotes the scalar curvature of $g_{\underline{\omega}}$. To remain consistent with this convention, we denote the torsion 2-form of $\phi$ by $\mathbf{T}$. 

Finally, when working in abstract index notation, we will use Roman indices $i,j,k, \ldots$ to refer to the $\mathbb{T}^3$ directions and Greek indices $\alpha, \beta, \gamma, \ldots$ to refer to the $X$ directions. So, for example, the 7-dimensional and 4-dimensional metrics are related by the equations $\boldsymbol{g}_{ij} = Q_{ij}$, $\boldsymbol{g}_{\alpha \beta} = g_{\alpha \beta}$ and $\boldsymbol{g}_{i\alpha} = 0$.

\subsection{The curvature tensors in 4 and 7 dimensions}

In this subsection we explain how the curvature tensor of $g_{\phi}$ on $X \times \mathbb{T}^3$ is made up of that of $g_{\underline{\omega}}$ on $X$ and terms involving $Q$ and its first and second derivatives. 

The first step is to compute the Christoffel symbols, $\boldsymbol{\Gamma}$ of $g_\phi$. The following is a simple calculation, the result of which we simply state.

\begin{lemma}
We have the following formulae for the Christoffel symbols of $g_{\phi}$:
\begin{align*}
\boldsymbol{\Gamma}^k_{ij} 
	&=
		0\\
\boldsymbol{\Gamma}^\gamma_{ij}
	&=
		- \frac{1}{2} g^{\gamma \alpha} \del_\alpha Q_{ij}\\
\boldsymbol{\Gamma}^{k}_{i\beta}
	&=
		\frac{1}{2}Q^{kl}\del_\beta Q_{il}\\
\boldsymbol{\Gamma}^k_{\alpha\beta}
	&=
		0\\
\boldsymbol{\Gamma}^\gamma_{\alpha\beta} 
	&= 
		\Gamma^{\gamma}_{\alpha\beta}
\end{align*}
\end{lemma}

From here one can directly compute the components of the curvature tensor of $g_\phi$. Again, we omit the calculations and simply state the results.

\begin{lemma}\label{curvature-formulas}
The components of the curvature tensor of $g_\phi$ are given by
\begin{align*}
\mathbf{R}_{ijk}^{\phantom{ijk}l}
	&=
		\frac{1}{4} \nabla^\beta Q_{ik}Q^{lp}\nabla_\beta Q_{pj}
		-
		\frac{1}{4}\nabla^\beta Q_{jk}Q^{lp}\nabla_\beta Q_{pi}\\
\mathbf{R}_{ijk}^{\phantom{ijk}\beta}
	&=
		0\\
\mathbf{R}_{ij\alpha}^{\phantom{ij\alpha}\beta}
	&=
		\frac{1}{4} \nabla^\beta Q_{jk}Q^{kl}\nabla_{\alpha}Q_{li}
		-
		\frac{1}{4} \nabla^\beta Q_{ik}Q^{kl}\nabla_{\alpha}Q_{lj}\\
\mathbf{R}_{i\beta k}^{\phantom{i\beta k}l}
	&=
		0\\
\mathbf{R}_{j\beta k}^{\phantom{j \beta k}\gamma}
	&=
		\frac{1}{2}\hat\nabla^\gamma \nabla_\beta Q_{jk}
		+
		\frac{1}{4}\nabla_{\beta}Q_{jp}Q^{pl}\nabla^\gamma Q_{lk}\\
\mathbf{R}_{\alpha \beta \gamma}^{\phantom{\alpha \beta \gamma} k}
	&=
		0\\
\mathbf{R}_{\alpha \beta \gamma}^{\phantom{\alpha \beta \gamma}\mu}
	&=
		R_{\alpha \beta \gamma}^{\phantom{\alpha \beta \gamma}\mu}
\end{align*}
\end{lemma}

\begin{corollary}\label{all-derivs-Rm-bound}
For any $\alpha_1, \ldots, \alpha_m$ we have
\[
\boldsymbol{\nabla}_{\alpha_1} \cdots \boldsymbol{\nabla}_{\alpha_m}
\mathbf{R}_{\alpha \beta \gamma}^{\phantom{\alpha \beta \gamma}\mu}
=
\nabla_{\alpha_1} \cdots \nabla_{\alpha_m} R_{\alpha \beta \gamma}^{\phantom{\alpha \beta \gamma}\mu}
\]
and hence $|\nabla^m \Rm|^2 \leq |\boldsymbol{\nabla}^m \mathbf{Rm}|^2$.
\end{corollary}

\begin{lemma}\label{HessianQ-bound}
\begin{equation*}
|\hat{\nabla}\diff Q|^2_Q  \leq  \frac{5}{4}\left(|\mathbf{Rm}|^2 + |\dd Q|_Q^4\right)
\end{equation*}
\end{lemma}

\begin{proof}
According to the fifth identity in Lemma \ref{curvature-formulas}, 
$$\hat\nabla^\alpha\nabla^\beta Q_{jk} = 2\mathbf{R}_{j\beta k}^{\;\;\;\;\; \;\alpha}
-\frac{1}{2} (\nabla_\beta Q Q^{-1}\nabla^\alpha Q)_{jk}$$
By diagonalizing $Q$ at a point and using normal coordinates $x_\alpha$ at this point for the metric $g_{\underline\omega}$, we have
\begin{align}
|\hat\nabla \dd Q|_Q^2 
& = (2\mathbf{R}_{j\beta k\alpha} - \frac{1}{2} (\nabla_\beta QQ^{-1}\nabla_\alpha)_{jk})
(2\RR_{p\;\; q}^{\;\; \beta \;\; \alpha} - \frac{1}{2}(\nabla^\beta QQ^{-1}\nabla^\alpha Q)_{pq})Q^{jp}Q^{kq} \nonumber\\
& =  \Big( 4 \RR_{j\beta k\alpha} \RR_{p\;\; q}^{\;\; \beta \;\; \alpha} 
- 2 \RR_{j\beta k \alpha} (\nabla^\beta QQ^{-1}\nabla_\alpha)_{pq}\Big)Q^{jp} Q^{kq}
+ \frac{1}{4} |\nabla^\alpha QQ^{-1}\nabla_\alpha Q|_Q^2\nonumber\\
&\leq |\mathbf{Rm}|^2  +
 \sum_{j,k, \alpha,\beta} 
\Big(\lambda_j^{-1}\lambda_k^{-1} \RR_{j\beta k\alpha}^2 + \lambda_j^{-1}\lambda_k^{-1} |\sum_p \nabla^\beta Q_{jp}\lambda_p^{-1}\nabla^\alpha Q_{pk}|^2 \Big)
+ \frac{1}{4} |\nabla^\alpha QQ^{-1}\nabla_\alpha Q|_Q^2\nonumber\\
& \leq \frac{5}{4} |\mathbf{Rm}|^2
 + \sum_{j,k,\alpha,\beta} \lambda_j^{-1}\lambda_k^{-1}\sum_p \lambda_p^{-1}|\nabla^\beta Q_{jp}|^2 \sum_q \lambda_q^{-1} |\nabla^\alpha Q_{qk}|^2 
 +\frac{1}{4} |\nabla^\alpha QQ^{-1}\nabla_\alpha Q|_Q^2\nonumber\\
 & = \frac{5}{4} |\mathbf{Rm}|^2 + |\dd Q|_Q^4 +\frac{1}{4} |\nabla^\alpha QQ^{-1}\nabla_\alpha Q|_Q^2\nonumber\\
 & \leq  \frac{5}{4}(|\mathbf{Rm}|^2 + |\dd Q|_Q^4)\qedhere
\end{align}
\end{proof}

%

\begin{lemma}\label{ricci-components}\label{Ricci-splitting}
The components of the Ricci tensor of $g_\phi$ are:
\begin{align*}
\mathbf{R}_{ij}
	&=
		- \frac{1}{2} (\hat\Delta Q)_{ij}\\
\mathbf{R}_{i\alpha}
	&=
		0\\
\mathbf{R}_{\alpha \beta}
	&=
		R_{\alpha \beta}
		-
		\frac{1}{4} \tr \left( 
		Q^{-1}\nabla Q \otimes Q^{-1} \nabla Q 
		\right)_{\alpha\beta}		
\end{align*}
The scalar curvature of $g_\phi$ is
\[
\mathbf{R} = R - \frac{1}{4}{|\diff Q|^2_Q}
\]
\end{lemma}
\begin{proof}
These formulae are obtained by direct calculation from the components of the full curvature tensor. We give the first as an example and suppress the details of the other calculations.
\begin{align*}
\mathbf{R}_{jk}
	&=
		\mathbf{R}_{ijk}^{\phantom{ijk} i} 
		+ 
		\mathbf{R}_{\alpha jk}^{\phantom{\alpha jk}^\alpha}\\
	&=
		\frac{1}{4} \nabla^\alpha Q_{ik}Q^{ip}\nabla_\alpha Q_{pj}
		-
		\frac{1}{4} \nabla^\alpha Q_{jk}Q^{ip} \nabla_\alpha Q_{pi}
		-
		\frac{1}{2} \hat\nabla^\alpha \nabla_\alpha Q_{jk}
		-
		\frac{1}{4} \nabla_\alpha Q_{jp} Q^{pl} \nabla^\alpha Q_{lk}\\
	&=
		- \frac{1}{2}( \hat{\Delta} Q)_{jk}
\end{align*}
since the first and last terms in the second line cancel and the second vanishes because $\det Q = 1$, which implies that  $\tr (Q^{-1} \nabla Q)=0$. 
\end{proof}
%
%

\subsection{The Levi-Civita connection on $\Lambda^+$ and the torsion 2-form}

The triple $\omega_i$ gives a framing of the bundle $\Lambda^+ \to X$ of self-dual 2-forms. In this framing, the Levi-Civita connection of $\Lambda^+$ is given by a matrix $a_{ij}$ of 1-forms:
\begin{equation}
\nabla \omega_i = a_{ij} \otimes \omega_j
\end{equation}
In this subsection we will explain how to determine the $a_{ij}$'s in terms of the torsion 2-form $\mathbf{T}$ of the $G_2$-structure. This will be important in the following subsection, when we control $\nabla \underline{\omega}$, $\nabla^2\underline{\omega}$ and $\boldsymbol{\nabla}\mathbf{T}$. 

To begin we give a purely 4-dimensional description of $\mathbf{T}$. We let $E_i \colon \Lambda^1 \to \Lambda^1$ be the operator defined by
\begin{equation}
E_i(\alpha) = - * (\alpha \wedge \omega_i)
\label{E-definition}
\end{equation}
When the triple $\underline{\omega}$ is hyperk\"ahler, the $E_i$'s are simply the hyperk\"ahler complex structures. In general, $E_i$ is skew-adjoint and $E_i^2 = -Q_{ii}$. With this in hand, we can describe the torsion 2-form. Define a triple of $1$-forms $\underline{\tau} = (\tau_1, \tau_2, \tau_3)$ by
\begin{equation}
\tau_i = - E_k(\diff Q Q^{-1})_{ik}
\label{tau-definition}
\end{equation}
\begin{lemma}
Given a hypersymplectic structure $\underline{\omega}$ on $X$, the torsion 2-form $\mathbf{T} = - \frac{1}{2} \diff^*\phi$ of the corresponding $G_2$-structure $\phi$ on $X \times \mathbb{T}^3$, defined as in \eqref{phi-from-omega}, is 
\begin{equation}
\mathbf{T} = - \frac{1}{2} \diff t^i \wedge \tau_i
\label{T-to-tau}
\end{equation}
\end{lemma}
\begin{proof}
Beginning from \eqref{dstar-phi} we have that
\[
\diff^*\phi 
	= Q_{jk} \diff t^k \wedge *\, \diff (Q^{ij} \omega_i)
	= - \diff t^k \wedge E_{i} (\diff Q^{-1}Q)_{ik}
	=  \diff t^k \wedge \tau_k
\]
which proves \eqref{T-to-tau}. 
\end{proof}

The main result of this subsection is the following. 

\begin{proposition}\label{LC-connection-matrix}
The Levi-Civita connection matrix is given by
\begin{equation}
a_{ij} = \frac{1}{2} (\diff QQ^{-1})_{ij} + (XQ^{-1})_{ij}
\label{formula-for-a}
\end{equation}
where $X$ is the matrix of 1-forms given by 
\[
X_{ij} = \frac{1}{2} \epsilon_{ijk} Q^{kl}\tau_l
\]
\end{proposition}

The proof will follow a series of lemmas.

\begin{lemma}
The connection matrix $a_{ij}$ is uniquely determined by the two equations
\begin{align}
aQ + Qa^t &= \diff Q \label{LC1}\\
E_ja_{ij} &= 0 \label{LC2}
\end{align}
\end{lemma}

\begin{proof}
It is a standard fact that $\nabla$ is the unique metric-compatible torsion-free connection on $\Lambda^+$. In the trivialisation given by the closed triple $\omega_i$, being metric-compatible is equivalent to \eqref{LC1}. Meanwhile, since $\diff \omega_i=0$, the torsion-free condition says that $a_{ij} \wedge \omega_j = 0$. Applying the Hodge star, we see that this is equivalent to \eqref{LC2}.
\end{proof}

\begin{lemma}[$Q$-twisted quaternion relations]
The operators $E_i$ satisfy
\begin{align}
E_iE_j 
	&= 
		- \epsilon_{ijk} Q^{kr} E_r - Q_{ij}\label{quaternion1}\\
E_iE_jE_k 
	&= 
		-Q_{ij} E_k  - Q_{jk} E_i + Q_{ik} E_j + \epsilon_{ijk}\label{quaternion2}
\end{align}
\end{lemma}

\begin{proof}
We begin with \eqref{quaternion1}. Write $\theta_i = Q^{-1/2}_{ij} \omega_j$. Then algebraically, the $\theta_i$ are a hyperk\"ahler triple. In other words, if we define $J_i (\alpha) = - * (\alpha \wedge \theta_i)$ then the $J_i$'s satisfy the quaternion relations $J_iJ_j = - \epsilon_{ijk} J_k - \delta_{ij}$. (The sign in front of the first term here is because we consider operators on covectors; the dual operators on tangent vectors satisfy $J_iJ_j = \epsilon_{ijk}J_k - \delta_{ij}$.) Now, since $\omega_i = Q^{1/2}_{ij} \theta_j$ we have that $E_i = Q^{1/2}_{ij}J_j$ and so
\[
E_iE_j 
	= 
		Q^{1/2}_{ip}Q^{1/2}_{jq} J_p J_q
	=
		Q^{1/2}_{ip}Q^{1/2}_{jq}(-\epsilon_{pqr}J_r - \delta_{pq})
\]
Now for any $3 \times 3$ matrix $A$ we have $\epsilon_{ijk}A_{ip}A_{jq}A_{kr} = \epsilon_{pqr} \det A$. Since $\det Q^{1/2} = 1$ we have $\epsilon_{pqr}Q^{1/2}_{ip}Q^{1/2}_{jq} = \epsilon_{ijt}Q^{-1/2}_{tr}$. This means that
\[
E_iE_j = -\epsilon_{ijt} Q^{-1/2}_{tr}J_r  - Q_{ij}
\]
Now $J_t = Q_{ts}^{-1/2} E_s$ which, after a relabelling of indices, gives \eqref{quaternion1}.

Equation \eqref{quaternion2} follows from two applications of \eqref{quaternion1} and the identity $\epsilon_{iqr}Q^{-1}_{pq}Q^{-1}_{rs} = \epsilon_{pst} Q_{ti}$, which holds since $\det Q=1$. We suppress the details.
\end{proof}

\begin{lemma}\label{Etau}
$\epsilon_{ijk} E_j \tau_k = \left( Q^{-1} \tau \right)_i$
\end{lemma}
\begin{proof}
We compute:
\begin{align*}
\epsilon_{ijk} E_j \tau_k 
	&=
	- \epsilon_{ijk} E_j E_r (\diff QQ^{-1})_{kr}\\
	&=
	 - \epsilon_{ijk} \left( - \epsilon_{jrp} Q^{-1}_{pq} E_q - Q_{jr}\right)
		(\diff QQ^{-1})_{kr}\\
	&=
	  \epsilon_{ijk} \epsilon_{jrp} Q^{-1}_{pq} E_q (\diff QQ^{-1})_{kr}
\end{align*}
since the term we have dropped comes from summing $\epsilon_{ijk}$ against $Q_{jr}(\diff Q Q^{-1})_{kr} = \diff Q_{kj}$ which is symmetric in $j,k$. Now
\begin{align*}
 \epsilon_{ijk} \epsilon_{jrp} Q^{-1}_{pq} E_q (\diff QQ^{-1})_{kr}
	&=
	- (\delta_{kp}\delta_{ir} - \delta_{kr}\delta_{pi}) 
		Q^{-1}_{pq} E_q (\diff Q Q^{-1})_{kr}\\
	&=
	- Q^{-1}_{kq} E_q (\diff QQ^{-1})_{ki}
\end{align*}
where the term we have dropped involves $\tr (\diff QQ^{-1})$ which vanishes since $\det Q =1$. Continuing,
\begin{align*}
- Q^{-1}_{kq} E_q (\diff QQ^{-1})_{ki}
	&=
	- E_q (Q^{-1}\diff Q Q^{-1})_{qi}\\
	&=
	- E_q(Q^{-1} \diff Q Q^{-1})_{iq}\\
	&=
	 (Q^{-1} \tau)_i \qedhere
\end{align*}
\end{proof}

\begin{remark}
This lemma is the 4-dimensional manifestation of the fact that the torsion 2-form in $X \times \mathbb{T}^3$ lies in one part of the decomposition of $\Lambda^2$ into $G_2$-irreducible summands.
\end{remark}

\begin{proof}[Proof of Proposition~\ref{LC-connection-matrix}]
Set $X_{ij} = \frac{1}{2}\epsilon_{ijk}Q^{kl}\tau_l$. Recall that we must show that 
\[
a = \frac{1}{2} \diff Q Q^{-1} + XQ^{-1}
\]
satisfies \eqref{LC1} and \eqref{LC2}. Equation \eqref{LC1} is immediate, since $X$ is skew-symmetric. To prove \eqref{LC2}, we compute:
\begin{equation}
(XQ^{-1})_{ij}
	=
		\frac{1}{2}\epsilon_{ipk}Q^{kl}\tau_l Q^{pj}
	=
		\frac{1}{2}\epsilon_{rjl}Q_{ir} \tau_l
\label{XQinv}
\end{equation}
(using $\det Q=1$) from which it follows that
\[
E_j(XQ^{-1})_{ij}
=
\frac{1}{2}\epsilon_{rjl} Q_{ir} E_j \tau_l
=
\frac{1}{2}\tau_i
\]
by Lemma~\ref{Etau}. Now
\[
E_ja_{ij} 
= \frac{1}{2} E_j (\diff QQ^{-1})_{ij} + \frac{1}{2}\tau_i
= -\frac{1}{2} \tau_i + \frac{1}{2} \tau_i 
=0
\]
as claimed.
\end{proof}

\subsection{Bounds on the derivatives of $\underline{\omega}$ and of torsion}

The aim of this subsection is to prove bounds on $\nabla \underline{\omega}$, $\nabla^2\underline{\omega}$, $\tau$, $\nabla \tau$ and $\boldsymbol{\nabla} \mathbf{T}$ purely in terms of $Q$ and its derivatives. The various inequalities and constants we obtain here are not meant to be sharp. The key point is that at each stage the control is purely in terms of $\tr Q$, $|\diff Q|_Q$ and $|\hat{\nabla}\diff Q|_Q$.

We begin with a lemma which will allow us to pass between various matrix norms. 

\begin{lemma}\label{matrix-norms}
Let $A$ be a $3\times3$ symmetric matrix of tensors. Then
\begin{enumerate}
\item
\[
\frac{3}{(\tr Q)^{2}}|A|_Q \leq |A| \leq \frac{\tr Q}{\sqrt{3}} |A|_Q
\]
\item
\[
|AQ^{-1}|^2 \leq \frac{1}{3}(\tr Q)^3 |A|_Q^2
\]
\end{enumerate}
\end{lemma}
\begin{proof}
These are pointwise estimates and so we can assume that $Q_{ij} = \lambda_i \delta_{ij}$ is diagonal at the point of interest, with $\lambda_1 \leq \lambda_2 \leq \lambda_3$. We have
\[
|A|^2_Q = \sum_{i,j} \lambda_i^{-1}\lambda_j^{-1} |A_{ij}|^2
\]
Since $(\tr Q)^2 \geq 3(\lambda_1\lambda_2 +\lambda_2\lambda_3 + \lambda_3 \lambda_1)$, it follows that $|A|_Q^2 \geq 3 (\tr Q)^{-2}|A|^2$. Meanwhile,
\begin{align*}
|A|_Q^2
	&=
		\sum_i \frac{1}{\lambda_i^2} |A_{ii}|^2+ \sum_{j\neq k} \frac{1}{\lambda_j\lambda_k} |A_{jk}|^2\\
	&=
		\left( \frac{\lambda_2\lambda_3}{\lambda_1} |A_{11}|^2
		+ \frac{\lambda_3\lambda_1}{\lambda_2}|A_{22}|^2
		+ \frac{\lambda_1\lambda_2}{\lambda_3}|A_{33}|^2
		\right)
		+ 2\lambda_1|A_{23}|^2+ 2\lambda_2|A_{31}|^2+ 2\lambda_3|A_{12}|^2\\
	&\leq
		(\lambda_2\lambda_3+\lambda_3\lambda_1 + \lambda_1\lambda_2)
		\left( \frac{1}{\lambda_1} + \frac{1}{\lambda_2} + \frac{1}{\lambda_3} \right)\sum_{i,j}|A_{ij}|^2\\
	&\leq 
		\frac{1}{9} (\tr Q)^4 |A|^2
\end{align*}
This proves the first pair of inequalities. To prove part 2, we note that 
\begin{align*}
|AQ^{-1}|^2 
	& = 
		\sum_i \frac{1}{\lambda_i^2}|A_{ii}|^2 
		+ \sum_{j<k} \left( \frac{\lambda_j}{\lambda_k}+ \frac{\lambda_k}{\lambda_j} \right)
		\frac{1}{\lambda_j\lambda_k}|A_{jk}|^2\\
	&\leq  
		\sum_i \frac{1}{\lambda_i^2}|A_{ii}|^2 
		+ \tr Q
		\sum_{j<k} \lambda_j\lambda_k 
		\sum_{j<k}\frac{1}{\lambda_j\lambda_k}|A_{jk}|^2\\
	&\leq  
		\sum_i \frac{1}{\lambda_i^2}|A_{ii}|^2  
		+\frac{1}{3}(\tr Q)^3\sum_{j\neq k}\frac{1}{\lambda_j\lambda_k}|A_{jk}|^2\\
	&\leq 
		\frac{1}{3}(\tr Q)^3 |A|_Q^2
\qedhere
\end{align*}
\end{proof}

\begin{lemma}\label{torsion-bounds}
We have the following bounds:
\begin{enumerate}
\item
$|\underline{\tau}|^2 \leq 2 \tr Q |\mathbf{T}|^2$.
\item
$|\mathbf{T}|^2 \leq \frac{3}{2}|\diff Q|^2_Q$.
\end{enumerate}
\end{lemma}
\begin{proof}
From \eqref{T-to-tau}, we have 
\[
|\mathbf{T}|^2 
	= 
		\frac{1}{2} Q^{ij} \left\langle \tau_i, \tau_j \right\rangle
	\geq 
		\frac{1}{2\tr Q} |\underline{\tau}|^2
\]
which proves the first inequality. For the second, the calculation is pointwise, so we assume that $Q_{ij}= \lambda_j\delta_{ij}$ with $\lambda_1\leq \lambda_2 \leq \lambda_3$. We compute, using $E_j^*E_j = - E_j^2 =  \lambda_j$, 
\[
|\tau_i|^2
	\leq
		3\sum_{j} |E_j(\diff QQ^{-1})_{ij}|^2
	= 
		3\sum_{j} \lambda_j |(\diff QQ^{-1})_{ij}|^2
	=
		3\sum_{j} \lambda^{-1}_j |\diff Q_{ij}|^2
\]
Now \eqref{T-to-tau} gives
\[
|\mathbf{T}|^2
	=
		\frac{1}{2} \sum_{i} \lambda_i^{-1} |\tau_i|^2
	\leq
		\frac{3}{2}\sum_{i,j} \lambda_i^{-1}\lambda_j^{-1} |\diff Q_{ij}|^2
	=
		\frac{3}{2} |\diff Q|^2_Q
\qedhere\]
\end{proof}

\begin{lemma}\label{nabla-omega-bound}
For any  hypersymplectic structure $\underline{\omega}$, 
\[
|\nabla \underline{\omega}|
	\leq 
		11 (\tr Q)^2 |\diff Q|_Q
\]
\end{lemma}

\begin{proof}
From \eqref{formula-for-a} we have
\begin{align*}
|a_{ij}|
	&\leq
		\frac{1}{2} \left|(\diff QQ^{-1})_{ij}\right|
		+
		\left|(XQ^{-1})_{ij}\right|\\
	&\leq
		\frac{1}{2\sqrt{3}}(\tr Q)^{3/2}|\diff Q|_Q
		+
		\frac{1}{2} \left| \epsilon_{pjq}Q_{ip}\tau_q\right|
\end{align*}
by Lemma~\ref{matrix-norms} for the first term, and \eqref{XQinv} for the second. For fixed $i,j$, $\epsilon_{pjq} Q_{ip} \tau_q$ is a sum of two terms of the form $Q_{ip}\tau_q$. Since $|Q_{ip}|^2 \leq \tr Q^2 \leq (\tr Q)^2$, we have $\left|\epsilon_{pjq} Q_{ip} \tau_q\right|\leq 2 \tr Q |\underline{\tau}|$. It follows that
\begin{equation}
\label{norm-a}
|a_{ij}| 
	\leq 
		\frac{1}{2\sqrt{3}} (\tr Q)^{3/2}\left| \diff Q \right|_Q + \tr Q |\underline{\tau}|
	\leq
		3(\tr Q)^{3/2} \left| \diff Q \right|_Q
\end{equation}
where we have used Lemma~\ref{torsion-bounds} which gives $|\underline{\tau}| \leq \sqrt{3} (\tr Q)^{1/2} |\diff Q|_Q$. Hence
\begin{align*}
|\nabla \underline{\omega}|^2
	& =
		\sum_{i,j,k} \left\langle a_{ij} \otimes \omega_j, a_{ik}
		\otimes \omega_k \right\rangle\\
	& =
		2\sum_{i,j,k} Q_{jk}\left\langle a_{ij},a_{ik} \right\rangle\\
	&\leq 
		 2 \tr Q \sum_{i,j} |a_{ij}|^2\\
	&\leq
		108 (\tr Q)^4 |\diff Q|^2_Q \qedhere
\end{align*}
\end{proof}

We next control $\nabla \underline{\tau}$. We begin with a lemma.

\begin{lemma}\label{nabla-E}
For all $\alpha \in \Lambda^1$, we have $(\nabla E_i)(\alpha) = a_{ij} \otimes E_j (\alpha)$.
\end{lemma}

\begin{proof}
The Levi-Civita connection commutes with the Hodge star. From this we see
\[
\nabla (E_i (\alpha)) 
	= 
		- \nabla(* (\alpha \wedge \omega_i)) 
	= 
		- *((\nabla \alpha) \wedge \omega_i 
		+ \alpha \wedge \nabla \omega_i) 
	=
		E_i(\nabla \alpha) + a_{ij} \otimes E_j(\alpha)
\]
This gives the result in view of $\nabla E_i(\alpha) = \nabla(E_i(\alpha)) - E_i(\nabla \alpha)$.
\end{proof}

\begin{lemma}\label{nabla-tau-bound}
There is a constant $C$ such that for any hypersymplectic structure $\underline{\omega}$, 
\[
|\nabla \underline{\tau}|
	\leq
		C (\tr Q)^2 \left( |\hat{\nabla} \diff Q|_Q + (\tr Q)^{5} |\diff Q|^2_Q\right)
\]		
\end{lemma}
\begin{proof}
By definition of $\underline{\tau}$ we have
\[
\nabla \tau_i
	=
		-(\nabla E_j)(\diff QQ^{-1})_{ij} - E_j(\nabla \diff Q Q^{-1})_{ij} + E_j(\diff Q Q^{-1}\otimes \diff Q Q^{-1})_{ij}
\]
We now bound each term separately, beginning with the first. By the formula \ref{nabla-E}
\begin{align*}
\left|(\nabla E_j) (\diff Q Q^{-1})_{ij}\right|^2
	&\leq
		9\sum_{j,k}|a_{jk}|^2 |E_k(\diff QQ^{-1})_{ij}|^2\\
	&\leq
		81 (\tr Q)^4 |\diff Q|^2_Q \sum_{j,k} |(\diff Q Q^{-1})_{ij}|^2\\
	&
		\leq 81 (\tr Q)^7 |\diff Q|^4_Q
\end{align*}
where in the second line we have used \eqref{norm-a} to bound $|a_{ij}|^2$ and the fact that $|E_k(\alpha)|^2 \leq \tr Q |\alpha|^2$, and we have written the sum explicitly since repeated indices  have disappeared. Then in the final line we applied Lemma~\ref{matrix-norms}. 

For the second term,
\begin{align*}
\left| E_j (\nabla \diff Q Q^{-1})_{ij}\right|^2
	&\leq
		3\tr Q \sum_j |(\nabla \diff Q Q^{-1})_{ij}|^2\\
	&\leq
		3 (\tr Q)^4|\nabla \diff Q|^2_Q
\end{align*}
(by Lemma~\ref{matrix-norms}). We recall equation \eqref{Hessians-relation} that 
\[
\nabla \diff Q_{ij} 
	= 
		(\hat{\nabla} \diff Q)_{ij} 
		+ 
		\frac{1}{2}Q^{pq} \left(
		\diff Q_{ip} \otimes \diff Q_{qj}
		+
		\diff Q_{qj} \otimes \diff Q_{ip}
		\right)
\]
It follows that
\[
|\nabla \diff Q - \hat{\nabla} \diff Q|^2
	\leq
	 9 \sum_{p,q}(Q^{pq})^2 \sum_{i,j,k,l} |\diff Q_{ik} \otimes \diff Q_{lj}|^2
	\leq  (\tr Q)^4 |\diff Q|^4
\]
Converting the left-hand side with Lemma~\ref{matrix-norms} gives
\[
|\nabla \diff Q - \hat{\nabla} \diff Q|^2_Q
	\leq 
		\frac{1}{81}(\tr Q)^{10} |\diff Q|^4_Q
\]
and so
\[
|\nabla \diff Q|_Q 
	\leq
		|\hat{\nabla} \diff Q|_Q + \frac{1}{9}(\tr Q)^{5} |\diff Q|^2_Q
\]
This implies that the second term in $\nabla \tau_i$ is controlled by
\[
\left| E_j (\nabla \diff Q Q^{-1})_{ij}\right|
	\leq 
		\sqrt{3}(\tr Q)^2 | \hat{\nabla} \diff Q|_Q + \frac{\sqrt{3}}{9}(\tr Q)^{7} |\diff Q|^2_Q
\]

Finally we come to the third term:
\begin{align*}
|E_j(\diff Q Q^{-1} \otimes \diff Q Q^{-1})_{ij}|^2
	&\leq
		3 \tr Q \sum_j |(\diff Q Q^{-1} \otimes \diff Q Q^{-1})_{ij}|^2\\
	&\leq 
		3 \tr Q \sum_{j,p,q} |(\diff QQ^{-1})_{ip}|^2|(\diff Q Q^{-1})_{qj}|^2\\
	&\leq
		3 \tr Q |\diff Q Q^{-1}|^4\\
	&\leq
		\frac{1}{3}( \tr Q )^{7} |\diff Q|^4_Q
\end{align*}
Putting the pieces together gives the result.
\end{proof}

\begin{lemma}\label{nabla-T-bound}
There is a constant $C$ such that for any hypersymplectic structure $\underline{\omega}$, we have
\[
|\boldsymbol{\nabla} \mathbf{T}|
	\leq
		C (\tr Q)^{8} \left( |\hat{\nabla}\diff Q|_Q + |\diff Q|^2_Q \right)
\]
\end{lemma}
\begin{proof}
Sicne $\mathbf{T} = -\frac{1}{2} \diff t^i \wedge \tau_i$ and since 
\begin{align*}
\boldsymbol{\nabla}\diff t^i 
	&=
		-\frac{1}{2} \diff t^j \otimes (\diff QQ^{-1})_{ji} 
		-\frac{1}{2} (\diff QQ^{-1})_{ij} \otimes \diff t^j\\
\boldsymbol{\nabla} \tau_i
	&=
		\nabla \tau_i + \frac{1}{2} \tau_i(\nabla Q_{jk}) \diff t^j \otimes \diff t^k
\end{align*}
we have that
\begin{align*}
\boldsymbol{\nabla}\mathbf{T}
	&=
		-\frac{1}{2} \boldsymbol{\nabla}(\diff t^i \otimes \tau_i - \tau_i \otimes \diff t^i)\\
	&=
		-\frac{1}{2}(\diff t^i \otimes \nabla \tau_i - \nabla \tau_i \otimes \diff t^i)\\
	&\quad\quad\quad
		+\frac{1}{4}\left(  
			(\diff QQ^{-1})_{ij}\otimes \diff t^j \otimes \tau_i
			+ \diff t^j \otimes (\diff Q Q^{-1})_{ji} \otimes \tau_i
		\right)\\
	&\quad\quad\quad\quad\quad
		-\frac{1}{4}\left(
			\tau_i \otimes (\diff Q Q^{-1})_{ij} \otimes \diff t^j
			+ \tau_i \otimes \diff t^j \otimes (\diff Q Q^{-1})_{ji}
		\right)\\
	&\quad \quad \quad \quad\quad\quad\quad
		-\frac{1}{2} \tau_i(\nabla Q_{jk})\left( 
			\diff t^i \otimes \diff t^j \otimes \diff t^k 
			- \diff t^j \otimes \diff t^k \otimes \diff t^i
		 \right)
\end{align*}
From this we deduce that
\begin{align*}
|\boldsymbol{\nabla} \mathbf{T}|
	&\leq
		\sum_i |\diff t^i| |\nabla \tau_i|
		+\sum_{i,j}|\diff t^{i}| |\tau_j| |(\diff Q Q^{-1})_{ij}|
		+ \sum_{i,j,k} |\tau_i(\nabla Q_{jk})||\diff t^i||\diff t^j||\diff t^k| \\
	&\leq 
		\tr Q |\nabla \underline{\tau}|
		+ 3 \tr Q |\underline{\tau}| |\diff QQ^{-1}|
		+ 3 (\tr Q)^3 |\underline{\tau}||\diff Q|\\
	&\leq
		C (\tr Q)^{8} \left( 
		 	|\hat{\nabla}\diff Q|_Q + |\diff Q|^2_Q
		\right)
		+ 3 (\tr Q)^{3} |\diff Q|^2_Q
		+ 3 (\tr Q)^{9/2} |\diff Q|^2_Q
\end{align*}
where in the second line we have used $|\diff t^i|^2 = Q^{ii} \leq (\tr Q)^2$ and in the third line we have used Lemmas~\ref{matrix-norms},~\ref{torsion-bounds} and~\ref{nabla-tau-bound}.
\end{proof}

\begin{lemma}\label{nabla2-omega-bound}
There is a constant $C$ such that for any hypersymplectic structure $\underline{\omega}$, 
\[
|\nabla^2 \underline{\omega}|
	\leq
		C (\tr Q)^8 \left( |\hat{\nabla} \diff Q|_Q +|\diff Q|_Q^2 \right)
\]
\end{lemma}

\begin{proof}
We compute, using the expression \eqref{formula-for-a} for $a_{ij}$, 
\begin{align*}
|\nabla^2 \omega_i|^2 
	&= 
		|\nabla a_{ij} \otimes \omega_j 
		+ a_{ij} \otimes a_{jk} \otimes \omega_k|^2\\
	&=
		\left| 
			\frac{1}{2} \nabla (\diff QQ^{-1})_{ij} \otimes \omega_j
			+ \frac{1}{2} \epsilon_{rjl} \nabla (Q_{ir}\tau_l) \otimes \omega_j
			+ a_{ij} \otimes a_{jk}\otimes \omega_k
		\right|^2\\
	&\leq
		\frac{3}{4} \left|
			\nabla (\diff Q Q^{-1})_{ij} \otimes \omega_j
		\right|^2
		+ \frac{3}{4} \left|
			\epsilon_{rjl}\nabla(Q_{ir} \tau_l)\otimes \omega_j
		\right|^2
		+ 3|a_{ij} \otimes a_{jk}\otimes \omega_k|^2
\end{align*}
The first term here is bounded by 
\begin{align*}
&\leq 
	\frac{3}{2}\tr Q \sum_j \left|(
		\left( 
			\nabla \diff Q) Q^{-1} - \diff Q Q^{-1}\otimes \diff Q Q^{-1}
		 \right)_{ij}
	\right|^2\\
&\leq
	 3\tr Q \sum_j | (\nabla\diff Q Q^{-1})_{ij}|^2 
	+ 3 \tr Q \sum_{j}
	|\sum_p (\diff QQ^{-1})_{ip}\otimes (\diff QQ^{-1})_{pj}|^2\\
&\leq
	 3\tr Q \sum_j | (\nabla\diff Q Q^{-1})_{ij}|^2 
	+ 9 \tr Q \sum_{j,p}
	|(\diff QQ^{-1})_{pj}|^2\sum_q |(\diff QQ^{-1})_{ip}|^2\\
&\leq
	 (\tr Q)^4 |\nabla \diff Q|^2_Q
	+ (\tr Q)^7 |\diff Q|^4_Q
\end{align*}
The second term is bounded by
\begin{align*}
&\leq
	3 \tr Q \sum_{j,l,r}|\nabla(Q_{ir}\tau_l)|^2\\
&\leq
	6 \tr Q \sum_{j,l,r} \left(  
		|\diff Q_{ir}|^2|\tau_l|^2 + (\tr Q)^2 |\nabla \tau_l|^2
	\right)\\
&\leq
	18 \tr Q |\diff Q|^2 |\underline{\tau}|^2 + 54 (\tr Q)^3 |\nabla \underline{\tau}|^2
\end{align*}
Finally, the third term is
\[
\left(\sum_{j,k} |a_{ij}||a_{jk}||\omega_k|\right)^2
	\leq
		2 \tr Q \left( \sum_{j,k} |a_{ij}||a_{jk}| \right)^2
\]
To complete the proof it suffices to combine these inequalities with those of Lemmas~\ref{torsion-bounds},~\ref{nabla-tau-bound} and~\ref{matrix-norms}, the formula~\eqref{formula-for-a} for $a_{ij}$ and the same switch from $\nabla \diff Q $ to $\hat\nabla \diff Q$ as in the proof of Lemma~\ref{nabla-T-bound}.
\end{proof}

\section{Evolution equations}
\label{evolution-equations}

In this section, we derive the evolution equations satisfied by $\tr Q$ and $|\diff Q|_Q^2$, in order to apply the maximum principle. We begin with the evolution equations satisfied by $Q$ and~$g_{\underline{\omega}}$. Lotay--Wei have computed the evolution of $g_{\phi}$ under the general $G_2$-Laplcian flow for a closed $G_2$-structure $\phi(t)$. We state their result here:

\begin{proposition}[Lotay--Wei \cite{Lotay--Wei}, equation (3.6)]
When a closed $G_2$-structure $\phi(t)$ evolves according to the $G_2$-Laplacian flow, the Riemannian metric $\boldsymbol{g}(t) = g_{\phi(t)}$ satisfies
\begin{equation}\label{full-G2-metric-flow}
\del_t \boldsymbol{g}_{ab}
	=
		-2 \mathbf{Ric}_{ab} - \frac{2}{3} |\mathbf{T}|^2\boldsymbol{g}_{ab}
		- 4\mathbf{T}_a^{\phantom{a}c}\mathbf{T}_{cb}
\end{equation}
\end{proposition}

Given the decomposition of $\boldsymbol{g}$ and $\mathbf{Ric}$ (see Lemma~\ref{Ricci-splitting}) this leads quickly to the 4-dimensional evolution equations. To write the equations, we first recall that $\det Q=1$, a condition that is obviously preserved under the flow. Now the subspace $\mathcal{S} = \{ Q \in \mathcal{P}: \det Q=1\}$ is totally geodesic. This implies that $\hat{\Delta}Q \in T_Q \mathcal{S}$. Write $\pr_Q \colon T_Q \mathcal{P} \to T_Q\mathcal{S}$ for the orthogonal projection. Explicitly, identifying $T_Q\mathcal{P}$ with symmetric matrices, we have $T_Q\mathcal{S} = \{A : \tr (Q^{-1}A)=0\}$ and $\pr_Q(A) = A - \frac{1}{3} \tr(Q^{-1}A)Q$. With this in hand we can state the evolution equations. They follow from direct manipulation of \eqref{full-G2-metric-flow} and so we suppress the details.

\begin{corollary}
When $\underline{\omega}(t)$ satisfies the hypersymplectic flow, $Q$ and $g_{\underline{\omega}}$ evolve according~to
\begin{align*}
\del_t Q 
	&= \hat{\Delta}Q + \pr_Q \left\langle \tau,\tau \right\rangle\\
\del_t g
	&=
		-2\Ric + \frac{1}{2} \left\langle \diff Q \otimes \diff Q \right\rangle_Q
		+ \tr (Q^{-1}\tau \otimes \tau)
		- \frac{2}{3} |\mathbf{T}|^2 g
\end{align*}
Here, $\left\langle \tau, \tau \right\rangle$ is the symmetric matrix with $(i,j)$-element $\left\langle \tau_i, \tau_j \right\rangle$, whilst
\begin{align*}
\left\langle \diff Q \otimes \diff Q \right\rangle_Q(u,v)
	&=
		Q^{ij}\nabla_u Q_{jk} Q^{kl}\nabla_v Q_{li}\\
\tr (Q^{-1}\tau\otimes \tau)(u,v)
	&=
		Q^{ij}\tau_{i}(u)\tau_j(v)
\end{align*}
\end{corollary}

\begin{remark}
If one ignores the terms involving torsion, we have precisely the flow studied in \cite{Muller}, in which the Ricci flow for $g$ and harmonic map flow for $Q$ are coupled.
\end{remark}

We now consider the heat operator acting on $\tr Q$.

\begin{proposition}\label{trQ-mp}
Under the hypersymplectic flow,
\[
(\del_t - \Delta) \tr Q \leq \frac{5}{3} |\mathbf{T}|^2 \tr Q
\]
It follows that if $|\mathbf{T}|$ is bounded for $t \in [0,s)$ with $s < \infty$ then $\tr Q$ is also bounded for $t\in[0,s)$. 
\end{proposition}
\begin{proof}
We have
\[
\del_t Q = \hat{\Delta} Q + \pr_Q \left\langle \tau, \tau \right\rangle
=
\Delta Q - \left\langle \diff Q, Q^{-1} \diff Q \right\rangle + \pr_Q \left\langle \tau, \tau \right\rangle
\]
It follows that
\[
(\del_t - \Delta)\tr Q
=
- \tr \left\langle \diff Q, Q^{-1} \diff Q \right\rangle 
+
|\underline{\tau}|^2 - \frac{1}{3} |\mathbf{T}|^2 \tr Q
\]
The stated inequality now follows from $\tr \left\langle \diff Q, Q^{-1} \diff Q \right\rangle \geq 0$ and the bound on $|\underline{\tau}|^2$ in Lemma~\ref{torsion-bounds}. Finally, the fact that a bound on $|\mathbf{T}|$ implies a bound on $\tr Q$ now follows from the maximum principle.
\end{proof}

Next we control the heat operator acting on $|\diff Q|^2_Q$. 

\begin{proposition}\label{dQ-mp}
There is a constant $C$ such that when $\underline{\omega}(t)$ evolves according to the hypersymplectic flow, we have
\[
(\del_t -\Delta) |\diff Q|^2_Q
	\leq
		-|\hat{\nabla}\diff Q|^2_Q
		-\frac{1}{16} |\diff Q|^4_Q
		+ C (\tr Q)^{21} |\mathbf{T}|^2 |\diff Q|^2_Q
\]
It follows that if $|\mathbf{T}|$ is bounded for $t \in [0,s)$ with $s <\infty$ then $|\diff Q|_Q$ is also bounded for $t \in [0,s)$.
\end{proposition}

\begin{proof}
A standard calculation from the theory of harmonic maps (was firstly carried out in \cite{ES}) gives
\begin{equation}
\frac{1}{2} \Delta |\diff Q|^2_Q
	=
		|\hat{\nabla}\diff Q|^2_Q
		+
		g^{\alpha \beta}\left\langle 
		\hat{\nabla}_\alpha \hat{\Delta}Q, \nabla_\beta Q
		\right\rangle_Q
		+
		R^{\alpha \beta} \left\langle 
		\nabla_\alpha Q,\nabla_\beta Q 
		\right\rangle_Q
		-
		K_{\mathcal{P}}
\label{Lap-norm-sq-dQ}
\end{equation}
where $K_{\mathcal{P}}$ is term involving the sectional curvature of $\mathcal{P}$:
\[
K_{\mathcal{P}} = \Rm_{\mathcal{P}}(
	\nabla_\alpha Q, \nabla_\beta Q, \nabla^\alpha Q, \nabla^\beta Q)
\]
Since $\mathcal{P}$ is non-positively curved $K_{\mathcal{P}} \leq 0$ and so we can safely discard this term. (The fact that $\mathcal{P} \cong \GL_+(3,\R)/\SO(3)$ is non-positively curved follows from the general theory of symmetric spaces. Alternatively one can calculate directly from the definition of the Levi-Civita connection given at the start of \S\ref{identities}.)
 
Meanwhile, writing momentarily $\del_t Q = V$ and $\del_t g_{\underline{\omega}} = N$, we have
\begin{equation}
\frac{1}{2}\del_t \left( |\diff Q|^2_Q \right)
	=
		g^{\alpha \beta}\left\langle 
		\hat{\nabla}_\alpha V, \nabla_\beta Q 
		\right\rangle_Q
		-\frac{1}{2} N^{\alpha \beta}
		\left\langle \nabla_\alpha Q, \nabla_\beta Q \right\rangle_Q
\label{dt-norm-sq-dQ}
\end{equation}
We expand the first term. To begin,
\[
\hat{\nabla} V 
	=
		\hat{\nabla}\hat{\Delta} Q
		+
		\hat{\nabla} \pr_Q\left\langle \tau,\tau \right\rangle
	=
		\hat{\nabla}\hat{\Delta} Q
		+
		\pr_Q \hat{\nabla} \left\langle \tau,\tau \right\rangle
\]
since $\mathcal{S}$ is totally geodesic. Moreover, $\nabla Q \in T_Q \mathcal{S}$ so, 
\begin{align*}
\left\langle 
\pr_Q\hat{\nabla} \left\langle \tau,\tau \right\rangle, 
\nabla Q \right\rangle_Q
	&=
		\left\langle
		 \hat{\nabla} \left\langle \tau, \tau \right\rangle , \nabla Q 
		 \right\rangle_Q\\
	&=
		\left\langle  
		\nabla \left\langle \tau, \tau \right\rangle
		- \left\langle \tau,\tau \right\rangle Q^{-1} \nabla Q,
		\nabla Q
		\right\rangle_Q
\end{align*}
This means that the first term of \eqref{dt-norm-sq-dQ} is
\begin{equation}
g^{\alpha\beta}\left\langle \hat{\nabla}_\alpha V, \nabla_\beta Q \right\rangle_Q
=
g^{\alpha \beta}
\left\langle 
\hat\nabla_\alpha \hat{\Delta} Q 
+
2 \left\langle \nabla_\alpha \tau, \tau \right\rangle 
-
\left\langle \tau,\tau \right\rangle Q^{-1} \nabla_\alpha Q
, \nabla_\beta Q\right\rangle_Q
\label{Q-term}
\end{equation}
Notice that (just as in the harmonic map flow) the first term of \eqref{Q-term} cancels the corresponding term in \eqref{Lap-norm-sq-dQ}. Meanwhile the last term of \eqref{Q-term} is nonpositive. To see this, write $M_\alpha= \left\langle \tau,\tau \right\rangle^{1/2}Q^{-1}\nabla_\alpha Q Q^{-1/2}$, then this term can be written as $-\sum_\alpha \tr (M_\alpha M_\alpha^t) \leq 0$.

The second term in \eqref{dt-norm-sq-dQ}, $-\frac{1}{2}N^{\alpha \beta} \left\langle \nabla_\alpha Q, \nabla_\beta Q \right\rangle_Q$, is
\begin{equation}		
R^{\alpha \beta} \left\langle 
\nabla_\alpha Q, \nabla_\beta Q 
\right\rangle_Q
	-
\frac{1}{4}
|\left\langle \nabla Q, \nabla Q \right\rangle_Q|^2
	-
\frac{1}{2} Q^{ij} \left\langle \tau_i(\nabla Q), \tau_j(\nabla Q) \right\rangle_Q
	+
\frac{1}{3} |\mathbf{T}|^2 |\diff Q|^2_Q
\label{metric-term}
\end{equation}
Just as for the coupled harmonic-Ricci flow, the Ricci term here cancels the corresponding term in \eqref{Lap-norm-sq-dQ}. The second term we bound as
\begin{align*}
|\left\langle \nabla Q, \nabla Q \right\rangle_Q|^2
	&=
		\sum_{\alpha,\beta}
		\left|\left\langle 
		\nabla^\alpha Q, \nabla^\beta Q 
		\right\rangle_Q\right|^2\\
	&\geq
		\sum_{\alpha}
		\left|\left\langle 
		\nabla^\alpha Q, \nabla^\alpha Q 
		\right\rangle_Q\right|^2\\
	&\geq
		\frac{1}{4} \left(  
		\sum_\alpha \left\langle 
		\nabla_\alpha Q,\nabla_\alpha Q
		\right\rangle_Q
		\right)^2\\
	&=
		\frac{1}{4} |\diff Q|^4_Q
\end{align*}
The third term in \eqref{metric-term} is nonpositive and so we can discard it. 

Putting this together, we obtain
\begin{equation}
\left( \del_t - \Delta \right)
|\diff Q|^2_Q
\leq
-2|\hat{\nabla}\diff Q|^2_Q
-
\frac{1}{8} |\diff Q|^4_Q
+
4 \left\langle 
\left\langle \nabla_\alpha \tau, \tau \right\rangle, \nabla^\alpha Q
\right\rangle_Q
+
\frac{2}{3} |\mathbf{T}|^2 |\diff Q|^2_Q
\label{intermediate-heat-inequality}
\end{equation}

It remains to control the third term in \eqref{intermediate-heat-inequality}. We have
\begin{align*}
\left\langle 
\left\langle \nabla_\alpha \tau, \tau \right\rangle, \nabla^\alpha Q
\right\rangle_Q
	&\leq
		\left|\left\langle \nabla_\alpha \tau, \tau \right\rangle\right|_Q
		\left| \nabla^\alpha Q \right|_Q\\
	&\leq
		(\tr Q)^2
		\left|\left\langle \nabla_\alpha \tau, \tau \right\rangle\right|
		\left| \nabla^\alpha Q \right|_Q\\
	&\leq
		(\tr Q)^2 |\nabla \underline{\tau}||\underline{\tau}|
		| \diff Q|_Q
\end{align*}
We now invoke the bounds on $\underline{\tau}$ and $\nabla \underline{\tau}$ derived in Lemmas~\ref{torsion-bounds} and~\ref{nabla-tau-bound}. This gives, for some absolute constant $C$,
\[
\left\langle 
\left\langle \nabla_\alpha \tau, \tau \right\rangle, \nabla^\alpha Q
\right\rangle_Q
	\leq
		2C(\tr Q)^{21/2}
		\left( | \hat{\nabla} \diff Q|_Q + |\diff Q|^2_Q \right)
		|\mathbf{T}| |\diff Q|_Q
\]
To complete the proof, we apply the ``Peter--Paul'' inequality, $2ab \leq \epsilon a^2 + \epsilon^{-1} b^2$. This gives
\begin{align*}
2C(\tr Q)^{21/2} |\hat{\nabla}\diff Q|_Q |\mathbf{T}||\diff Q|_Q
	\leq
		\frac{1}{4}|\hat{\nabla}\diff Q|^2_Q
		+
		4 C^2(\tr Q)^{21} |\mathbf{T}|^2|\diff Q|^2_Q\\
2C(\tr Q)^{21/2} |\mathbf{T}| |\diff Q|^3_Q
	\leq
		\frac{1}{64} |\diff Q|^4_Q
		+
		64C^2 (\tr Q)^{21} |\mathbf{T}|^2 |\diff Q|^2_Q
\end{align*}
This means that
\[
\left\langle 
\left\langle \nabla_\alpha \tau, \tau \right\rangle, \nabla^\alpha Q
\right\rangle_Q
	\leq
		\frac{1}{4}|\hat{\nabla}\diff Q|^2_Q
		+
		\frac{1}{64} |\diff Q|^4_Q
		+
		68 C^2 (\tr Q)^{21} |\mathbf{T}|^2 |\diff Q|^2_Q
\]
Substituting this into \eqref{intermediate-heat-inequality} completes the proof of the heat inequality. The conclusion that a bound on $|\mathbf{T}|$ yields a bound on $|\diff Q|_Q$ (in finite time) then follows from the maximum principle (given that $\tr Q$ is also bounded, by Proposition~\ref{trQ-mp}). 
\end{proof}

\section{Control of the $L^2$-norm of curvature}
\label{L2-curvature}

\subsection{Overview of the proof}

In this section we will show that a bound on $|\mathbf{T}|$ implies a bound on the energy of the metric $g_{\underline{\omega}}$, i.e., the $L^2$-norm of its curvature, at least in finite time. The inspiration for this is an article of Miles Simon \cite{Simon} which proves an analogous result for the 4-dimensional Ricci flow:

\begin{theorem}[Simon \cite{Simon}]\label{Simon-theorem}
Let $g(t)$ be a solution to Ricci flow on a compact 4-manifold and time interval $t \in [0,s)$ with $s< \infty$. Suppose that for all $t$, $|R(g(t)|\leq \beta/2$. Then there is a constant $C$ depending only on $\beta$, $s$ and the initial data such that 
\begin{enumerate}
\item
For all $t \in [0,s)$, 
\[
\int |\Rm|^2 < C
\]
\item
\[
\int_0^s \diff t \int |\Ric|^4 + |\nabla \Ric|^2 < C
\]
\end{enumerate}
\end{theorem}

We begin by briefly reviewing Simon's argument, before going on to explain the additional complications which arise in our situation. There are three key ingredients:
\begin{enumerate}
\item
In dimension~4, the Chern--Gauss--Bonnet theorem says that for any compact Riemannian 4-manifold $X$, 
\begin{equation}
32\pi^2 \chi(X) = \int |\Rm|^2 - 4|\Ric|^2 + R^2
\label{CGB}
\end{equation}
\item
There is a constant $C>0$ such that along the Ricci flow, 
\begin{equation}
\frac{\diff}{\diff t} \int |\Ric|^2
	\leq
		\int -|\nabla \Ric|^2 + 2C |\Rm||\Ric|^2
\label{RF-L2-Ricci}
\end{equation}
\item
Assume that $|R|<\beta/2$ along the Ricci flow. Then there is a constant $C>0$ (depending on $\beta$) such that along the Ricci flow,
\begin{equation}
\frac{\diff}{\diff t} \int \frac{|\Ric|^2}{R+\beta}
	\leq
		\int - \frac{4}{\beta^2} |\Ric|^4 + C |\Rm||\Ric|^2 + |\nabla \Ric|^2
\label{RF-Ricc4}
\end{equation}
\end{enumerate}

From here Simon argues as follows. Write $I$ for the integrand in the Chern--Gauss--Bonnet formula \eqref{CGB}. Let $\epsilon >0$. In what follows, $K$ denotes a positive constant which may depend on $\epsilon$ and may change from line to line, but is otherwise absolute. We have
\begin{align*}
2C|\Rm||\Ric|^2
	&\leq
		\epsilon |\Ric|^4 + K|\Rm|^2\\
	&\leq
		\epsilon |\Ric|^4 + K( I + 4|\Ric|^2)\\
	&\leq
		2\epsilon |\Ric|^4 + KI + K
\end{align*}
It follows from \eqref{RF-L2-Ricci} that
\[
\frac{\diff}{\diff t}\int |\Ric|^2
\leq
\int
 -|\nabla \Ric|^2 +2\epsilon |\Ric|^4 + K 
\]
Similarly, starting with \eqref{RF-Ricc4} we do the same, this time absorbing the positive $|\Ric|^4$ term into the negative $-\frac{4}{\beta^2}|\Ric|^4$. This means that there is a constant $K$ such that
\[
\frac{\diff}{\diff t} \int \frac{|\Ric|^2}{R+\beta}
\leq
\int
-\frac{3}{\beta^2} |\Ric|^4 + |\nabla \Ric|^2 + K
\]
Now putting these two inequalities togethers, with a small enough choice of $\epsilon >0$, gives a constant $K$ such that
\[
\frac{\diff}{\diff t}\int
|\Ric|^2 + \frac{|\Ric|^2}{2(R+\beta)}
\leq
\int -\frac{1}{\beta^2} |\Ric|^4 - \frac{1}{2}|\nabla \Ric|^2 + K
\]
Integrating this over $t \in [0,s)$ proves a uniform bound on $\int |\Ric|^2$ as well as the second part of Simon's Theorem~\ref{Simon-theorem}. The Chern--Gauss--Bonnet theorem then turns the bound on $\int |\Ric|^2$ into a uniform bound on $\int |\Rm|^2$, completing the proof.

Our main result in this section is an analogue of Simon's theorem for the hypersymplectic flow:

\begin{theorem}[$L^2$-estimate of 7-dimensional Ricci curvature]\label{Ricci-L2-bound}
Let $\underline{\omega}(t)$ be a solution to the hypersymplectic flow on a compact 4-manifold $X$ and time interval $[0,s)$ with $s<\infty$. Suppose that $|\mathbf{T}|^2 < \beta/2$ at all times. Then there exists a constant $C$ which depends only on $\beta$, $s$ and the initial data such that
\begin{enumerate}
\item
For all $t \in [0,s)$, 
\[
\int_{X\times \mathbb{T}^3} |\mathbf{Ric}|^2
\leq
C
\]
\item
\[
\int_0^s \diff t
\int_{X\times \mathbb{T}^3}
\left( 
|\mathbf{Ric}|^4 + |\boldsymbol{\nabla} \mathbf{Ric}|^2
+
|\mathbf{Ric}|^2|\hat{\nabla}\diff Q|^2_Q
+
|\hat{\nabla}\diff Q|^2_Q
\right) \boldsymbol{\mu}
\leq
C
\]
\end{enumerate}
\end{theorem}

Using the relationship between the 4- and 7-dimensional curvature tensors (Lemma~\ref{Ricci-splitting}), we have
\[
|\Ric|^2 \leq |\mathbf{Ric}|^2 + \frac{1}{8}|\diff Q|^4_Q
\]
Now $|\diff Q|_Q$ is bounded along the flow (by Proposition~\ref{dQ-mp}), and the total volume is bounded by a constant depending only on the cohomology classes of the $\omega_i$:
\[
\int \mu \leq \frac{1}{3}\int \tr Q \mu = \frac{1}{6}  \int \omega_1^2+\omega_2^2+\omega_3^2
\]
It follows that the $L^2$-bound on $\mathbf{Ric}$ gives an $L^2$-bound on the 4-manifold Ricci curvature $\Ric$ and hence, by Chern--Gauss--Bonnet, on $\Rm$ as well. We can similarly translate the seond part of Theorem~\ref{Ricci-L2-bound} in to 4-dimensional quantities. The result is as follows.

\begin{corollary}[4-dimensional Energy bound]\label{4D-energy-bound}
Under the same assumptions as in Theorem~\ref{Ricci-L2-bound}, there exists a constant $C$, again depending only on $\beta$, $s$ and the initial data such that
\begin{enumerate}
\item
For all $t \in [0,s)$, 
\[
\int_X |\Rm|^2 \leq C
\]
\item
\[
\int_0^s \diff t \int_X
 |\Ric|^4 + |\hat{\Delta}Q|^4_Q + |\nabla \Ric|^2 + |\hat{\nabla}\diff Q|^2_Q
+
|\Ric|^2 |\hat{\nabla}\diff Q|^2_Q + |\hat{\Delta}Q|^2_Q|\hat{\nabla}\diff Q|^2_Q
 \mu \leq C
\]
\end{enumerate}
\end{corollary}

We now give the proof of Theorem~\ref{Ricci-L2-bound}, along the lines of Simon's argument. The crux is to find a series of differential inequalities which play the role of \eqref{RF-L2-Ricci} and \eqref{RF-Ricc4}. We will state the inequalities here and defer their proofs until the subsequent sections. Note that in the following Proposition, all integrals are over the 7-manifold $X \times \mathbb{T}^3$, and we have suppressed the volume form~$\boldsymbol{\mu}$ (which must be remembered when taking the time derivative, of course!).

\begin{proposition}\label{differential-inequalities}
Let $\underline{\omega}(t)$ be a solution to the hypersymplectic flow on a compact 4-manifold and time interval $t\in [0,s)$. Suppose moreover that $|\mathbf{T}|^2 \leq \beta/2$ for all $t$. Then there exists a constant $C$, depending only on $\beta$, $s$ and the initial data, such that for all $t \in [0,s)$,
\begin{enumerate}
\item
$
\frac{\diff}{\diff t} \int |\mathbf{Ric}|^2
	\leq
		\int - |\boldsymbol{\nabla}\mathbf{Ric}|^2
			+ C\left( 
			|\mathbf{Rm}||\mathbf{Ric}|^2 
			+ |\mathbf{Ric}|^2 
			+ |\hat{\nabla}\diff Q|^2_Q
			+ 1
			\right)
$
\item
$
\frac{\diff}{\diff t} \int \frac{|\mathbf{Ric}|^2}{\mathbf{R}+\beta}
	\leq
		\int - \frac{4}{\beta^2} |\mathbf{Ric}|^4 
			+|\boldsymbol{\nabla}\mathbf{Ric}|^2
			+ C\left( 
			|\mathbf{Rm}||\mathbf{Ric}|^2 
			+ |\mathbf{Ric}|^2|\hat{\nabla}\diff Q|^2_Q
			+ |\mathbf{Ric}|^2 
			\right)
$
\item
$
\frac{\diff}{\diff t} \int |\mathbf{Ric}|^2|\diff Q|^2_Q
	\leq
		\int
		-\frac{1}{2}|\mathbf{Ric}|^2|\hat{\nabla}\diff Q|^2_Q
		+C\left(  
		|\mathbf{Rm}||\mathbf{Ric}|^2 
		+|\hat{\nabla}\diff Q|^2_Q
		+|\boldsymbol{\nabla}\mathbf{Ric}|^2
		+|\mathbf{Ric}|^2
		+ 1
		\right)
$
\item
$
\frac{\diff}{\diff t} \int |\diff Q|^2_Q
	\leq 
		\int -|\hat{\nabla}\diff Q|^2_Q + C
$
\end{enumerate}
\end{proposition}

\begin{proof}[Proof of Theorem~\ref{Ricci-L2-bound}, assuming Proposition~\ref{differential-inequalities}]
Throughout, $C$ denotes a constant which depends only on $\beta$, $s$ and the initial data, but which may change from line to line. Arguing exactly as in Simon's proof, to deal with $|\mathbf{Rm}||\mathbf{Ric}|^2$ via Chern--Gauss--Bonnet, we have 
\[
\frac{\diff}{\diff t} \int \frac{|\mathbf{Ric}|^2}{\mathbf{R}+\beta}
	\leq
		\int 
		-\frac{3}{\beta^2}|\mathbf{Ric}|^4 
		+ C\left( 
		|\mathbf{Ric}|^2|\hat{\nabla}\diff Q|^2 
		+ |\boldsymbol{\nabla}\mathbf{Ric}|^2
		+1
		 \right)
\]
It follows that for a suitable choice of $A_1>0$ (and again using the Chern--Gauss--Bonnet trick) we have
\[
\frac{\diff}{\diff t} \int
\frac{|\mathbf{Ric}|^2}{\mathbf{R}+\beta}
+ A_1 |\mathbf{Ric}|^2|\diff Q|^2_Q
	\leq	
		\int
		-\frac{2}{\beta^2}|\mathbf{Ric}|^4 
		- |\mathbf{Ric}|^2|\hat{\nabla}\diff Q|^2_Q
		+C\left(  
		|\boldsymbol{\nabla}\mathbf{Ric}|^2 + |\hat{\nabla}\diff Q|^2_Q
		+1
		\right)
\]
Now for a suitable choice of $A_2>0$ we have
\begin{multline*}
\frac{\diff}{\diff t} \int
\frac{|\mathbf{Ric}|^2}{\mathbf{R}+\beta}
+ A_1 |\mathbf{Ric}|^2|\diff Q|^2_Q
+ A_2 |\mathbf{Ric}|^2\\
	\leq
		\int
		-\frac{1}{\beta^2}|\mathbf{Ric}|^4 
		- |\mathbf{Ric}|^2|\hat{\nabla}\diff Q|^2_Q
		- |\boldsymbol{\nabla}\mathbf{Ric}|^2
		+C\left( 
		|\hat{\nabla}\diff Q|^2_Q + 1 \right)
\end{multline*}
Finally, for a suitable choice of $A_3>0$,
\begin{multline}\label{big-differential-inequality}
\frac{\diff}{\diff t} \int
\frac{|\mathbf{Ric}|^2}{\mathbf{R}+\beta}
+ A_1 |\mathbf{Ric}|^2|\diff Q|^2_Q
+ A_2 |\mathbf{Ric}|^2
+ A_3 |\diff Q|^2_Q\\
	\leq
		\int
		-\frac{1}{\beta^2}|\mathbf{Ric}|^4 
		- |\mathbf{Ric}|^2|\hat{\nabla}\diff Q|^2_Q
		- |\boldsymbol{\nabla}\mathbf{Ric}|^2
		- |\diff Q|^2_Q
		+C
\end{multline}
Theorem~\ref{Ricci-L2-bound} follows by integrating \eqref{big-differential-inequality} over $t \in [0,s)$, together with the fact that both the volume and $|\diff Q|_Q$ are uniformly bounded. 
\end{proof}

We now give the proofs of the inequalities in Proposition~\ref{differential-inequalities}. At times the formulae involved may appear intimidating on the page, but the arguments involve nothing more than integration by parts and careful bookkeeping.

\subsection{Evolution of Ricci curvature under the $G_2$-Laplacian flow}

The first two inequalities in Proposition~\ref{differential-inequalities} will follow from more general inequalities which hold for an arbitrary $G_2$-Laplacian flow of closed $G_2$-structures. We begin by recalling the 7-dimensional evolution equations for the volume $\boldsymbol{\mu}$, scalar curvuatre $\mathbf{R}$, and Ricci tensor $\mathbf{Ric}$ derived by Lotay--Wei \cite{Lotay--Wei}.
\begin{align}
\del_t \boldsymbol{\mu}
	&=
		\frac{2}{3} |\mathbf{T}|^2 \boldsymbol{\mu}\label{vol-evolution}\\
\del_t \mathbf{R}
	&=
		\boldsymbol{\Delta}\mathbf{R}
		-
		4\boldsymbol{\nabla}^b\boldsymbol{\nabla}^a
		(\mathbf{T}_a^{\phantom{a}c}\mathbf{T}_{cb})
		+
		2|\mathbf{Ric}|^2
		-
		\frac{2}{3}\mathbf{R}^2
		+
		4 \mathbf{R}^{ab}\mathbf{T}_a^{\phantom{a}c}\mathbf{T}_{cb}\label{scalar-heat}\\
\del_t \mathbf{R}_{ab}
	&=
		\boldsymbol{\Delta}_L \left( 
		\mathbf{R}_ab 
		+ 
		\frac{1}{3} |\mathbf{T}|^2 
		+ 
		2 \mathbf{T}_a^{\phantom{a}c}\mathbf{T}_{cb}
		\right)
		\nonumber\\
	&\quad\quad		
		-\frac{1}{3} 
		\boldsymbol{\nabla}_a\boldsymbol{\nabla}_b |\mathbf{T}|^2
		-2\left( 
		 \boldsymbol{\nabla}_a\boldsymbol{\nabla}^c
		 \left( \mathbf{T}_c^{\phantom{c}d}\mathbf{T}_{db} \right)
		 +
		 \boldsymbol{\nabla}_b\boldsymbol{\nabla}^c
		 \left( \mathbf{T}_c^{\phantom{c}d}\mathbf{T}_{da} \right)
		 \right)\label{7D-Ricci-evolution}
\end{align}
where $\boldsymbol{\Delta}_L \eta_{ab} = \boldsymbol{\Delta} \eta_{ab} - \mathbf{R}_a^{\phantom{a}c}\eta_{cb} - \mathbf{R}_b^{\phantom{b}c}\eta_{ca} + 2 \mathbf{R}_{cabd} \eta^{dc}$ is the Lichnerowicz Laplacian. 

Using \eqref{7D-Ricci-evolution} one finds the heat equation satisfied by $|\mathbf{Ric}|^2$:
\begin{multline}
(\del_t - \boldsymbol{\Delta}) |\mathbf{Ric}|^2
	=
		- 2 |\boldsymbol{\nabla}\mathbf{Ric}|^2
		+ 4 \mathbf{R}_{cabd}\mathbf{R}^{ab}\mathbf{R}^{dc}
		+ 4 \mathbf{R}^{ab}\boldsymbol{\Delta}\left(
		\mathbf{T}_a^{\phantom{a}c}\mathbf{T}_{cb}
		\right)\\
		+\frac{2}{3}\mathbf{R}\boldsymbol{\Delta}|\mathbf{T}|^2	
		- 4\mathbf{R}^{ab}\left(	
		 \boldsymbol{\nabla}_a\boldsymbol{\nabla}^c
		 \left( \mathbf{T}_c^{\phantom{c}d}\mathbf{T}_{db} \right)
		 +
		 \boldsymbol{\nabla}_b\boldsymbol{\nabla}^c
		 \left( \mathbf{T}_c^{\phantom{c}d}\mathbf{T}_{da} \right)
		 \right)\\
		 -\frac{2}{3}\mathbf{R}^{ab}
		 \boldsymbol{\nabla}_a\boldsymbol{\nabla}_b |\mathbf{T}|^2
		 + 8 \mathbf{R}_{cabd}\mathbf{R}^{ab}
		 \mathbf{T}^{ce}\mathbf{T}_e^{\phantom{e}d}
		 +
		 \frac{4}{3}|\mathbf{Ric}|^2|\mathbf{T}|^2
\label{Ricci-heat}
\end{multline}
We will also need the following inequality, satisfied by any closed $G_2$-structure. Since $|\mathbf{Ric}|^2 \geq \frac{1}{7}|\mathbf{R}|^2$ we have
\begin{equation}\label{torsion-bounded-by-Ric}
|\mathbf{T}|^2 \leq \sqrt{7} |\mathbf{Ric}|
\end{equation}

We now turn to our differential inequalities. Note that in the Lemmas below, the volume form $\boldsymbol{\mu}$ is omitted to ease the notation, but during the calculation of the time derivatives we have to take it into consideration. 

\begin{lemma}
Let $\phi(t)$ be a path of closed $G_2$-forms solving the $G_2$-Laplacian flow. Then
\begin{equation}
\frac{\diff}{\diff t}
\int |\mathbf{Ric}|^2
	\leq
		\int\left( 
		-|\boldsymbol{\nabla}\mathbf{Ric}|^2
		+ 28 |\mathbf{Rm}||\mathbf{Ric}|^2
		+ 18 |\mathbf{T}|^2
		|\boldsymbol{\nabla}\mathbf{T}|^2
		+ 2|\mathbf{Ric}|^2|\mathbf{T}|^2
		\right)
\label{7D-Ricci-diff-ineq}
\end{equation}
\end{lemma}
\begin{proof}
From \eqref{vol-evolution}, \eqref{Ricci-heat}, integration by parts and the contracted Bianchi identity we get
\begin{align*}
\frac{\diff}{\diff t} \int |\mathbf{Ric}|^2
	&=
		\int \left( \del_t - \boldsymbol{\Delta} \right)|\mathbf{Ric}|^2
		+ \frac{2}{3} |\mathbf{Ric}|^2|\mathbf{T}|^2\\
	&=
		\int - 2 |\boldsymbol{\nabla}\mathbf{Ric}|^2
		+ 4 \mathbf{R}_{cabd}\mathbf{R}^{ab}\mathbf{R}^{dc}
		- 4 \boldsymbol{\nabla}^c\mathbf{R}^{ab}
		\boldsymbol{\nabla}_c\left(
		\mathbf{T}_a^{\phantom{a}d}\mathbf{T}_{db}
		\right)\\
	&\quad\quad
		-\frac{1}{3}\boldsymbol{\nabla}^c\mathbf{R}
		\boldsymbol{\nabla}_c|\mathbf{T}|^2	
		+4\boldsymbol{\nabla}^b\mathbf{R}	
		\boldsymbol{\nabla}^c
		( \mathbf{T}_c^{\phantom{c}d}\mathbf{T}_{db})
		 + 8 \mathbf{R}_{cabd}\mathbf{R}^{ab}
		 \mathbf{T}^{ce}\mathbf{T}_e^{\phantom{e}d}
		 +
		 2|\mathbf{Ric}|^2|\mathbf{T}|^2
\end{align*}
Next we apply Kato's inequality, $|\boldsymbol{\nabla}|\mathbf{T}|| \leq | \boldsymbol{\nabla}\mathbf{T}|$ and the identity $\nabla \mathbf{R} = -2 |\mathbf{T}| \boldsymbol{\nabla}|\mathbf{T}|$ to obtain
\begin{align*}
\frac{\diff}{\diff t} \int |\mathbf{Ric}|^2
	&\leq
		\int - 2 |\boldsymbol{\nabla}\mathbf{Ric}|^2
		+ 4|\mathbf{Rm}||\mathbf{Ric}|^2
		+8|\boldsymbol{\nabla}\mathbf{Ric}||\mathbf{T}|
		|\boldsymbol{\nabla}\mathbf{T}|
		+\frac{4}{3} |\mathbf{T}|^2|\boldsymbol{\nabla}\mathbf{T}|^2\\
	&\quad\quad
		+8|\mathbf{T}|^2|\boldsymbol{\nabla}\mathbf{T}|^2
		+ 8 |\mathbf{Rm}||\mathbf{Ric}||\mathbf{T}|^2
		+ 2 |\mathbf{Ric}|^2|\mathbf{T}|^2\\
	&\leq	
		\int - |\boldsymbol{\nabla}\mathbf{Ric}|^2
		+ (4 + 8 \sqrt{7})|\mathbf{Rm}||\mathbf{Ric}|^2
		+ \frac{52}{3} |\mathbf{T}|^2|\boldsymbol{\nabla}\mathbf{T}|^2
		+ 2 |\mathbf{Ric}|^2|\mathbf{T}|^2
\end{align*}
where in the last step we have used Cauchy--Schwarz and \eqref{torsion-bounded-by-Ric}. The stated inequalitiy now follows (by replacing coefficients by larger integers). 
\end{proof}

\begin{corollary}
Let $\underline{\omega}(t)$ be a solution to the hypersymplectic flow on a compact 4-manifold and time interval $t\in [0,s)$. Suppose moreover that $|\mathbf{T}|^2 \leq \beta/2$ for all $t$. Then there exists a constant $C$, depending only on $\beta$, $s$ and the initial data, such that for all $t \in [0,s)$
\[
\frac{\diff}{\diff t} \int |\mathbf{Ric}|^2
	\leq
		\int - |\boldsymbol{\nabla}\mathbf{Ric}|^2
			+ C\left( 
			|\mathbf{Rm}||\mathbf{Ric}|^2 
			+ |\mathbf{Ric}|^2 
			+ |\hat{\nabla}\diff Q|^2_Q
			+ 1
			\right)
\]
\end{corollary}

\begin{proof}
By Proposition~\ref{trQ-mp}, $\tr Q$ is bounded. By Proposition~\ref{dQ-mp}, $|\diff Q|_Q$ is bounded, so $|\mathbf{T}|$ is bounded. It then follows from Lemma~\ref{nabla-T-bound} that
\[
|\boldsymbol{\nabla}\mathbf{T}| \leq C \left( |\hat{\nabla}\diff Q|^2_Q + 1 \right)
\]
The result now follows by substituting these bounds in to \eqref{7D-Ricci-diff-ineq}.
\end{proof}

We now turn to the second inequality in Proposition~\ref{differential-inequalities}.

\begin{lemma}
Let $\phi(t)$ be a path of closed $G_2$-forms solving the $G_2$-Laplacian flow and assume moreover that there exists $\beta >0$ such that for all $t \in [0,s)$ we have $\mathbf{R}+\beta >0$. Then
\begin{align}
\frac{\diff}{\diff t} \int \frac{|\mathbf{Ric}|^2}{\mathbf{R}+\beta}
	&\leq
		\int -\frac{|\mathbf{Ric}|^4}{(\mathbf{R}+\beta)^2}
		+\frac{28}{\mathbf{R}+\beta} |\mathbf{Rm}||\mathbf{Ric}|^2
		+\frac{3|\mathbf{T}|^4+2\beta|\mathbf{T}|^2}{(\mathbf{R}+\beta)^2}
		|\mathbf{Ric}|^2\nonumber\\
	&\quad\quad\quad\quad\quad\quad\quad\quad
		+640(1+\beta)
		\frac{\mathbf{R}^2+\beta^2+1}{(\mathbf{R}+\beta)^4}
		|\mathbf{Ric}|^2|\boldsymbol{\nabla} T|^2
		+ |\boldsymbol{\nabla}\mathbf{Ric}|^2
		\label{7D-Ricc4-bound}
\end{align}
\end{lemma}
\begin{proof}
We compute:
\begin{align*}
\frac{\diff}{\diff t} \int \frac{|\mathbf{Ric}|^2}{\mathbf{R}+ \beta}
	&=
		\int \left( \del_t - \boldsymbol{\Delta} \right)
		\frac{|\mathbf{Ric}|^2}{\mathbf{R}+ \beta}
		+
		\frac{2|\mathbf{Ric}|^2|\mathbf{T}|^2}{3(\mathbf{R}+ \beta)}\\
	&=
		\int 
		\frac{\left( \del_t - \boldsymbol{\Delta} \right)|\mathbf{Ric}|^2}
		{\mathbf{R}+\beta}
		+
		\frac{2 \left\langle \boldsymbol{\nabla}|\mathbf{Ric}|^2,
		 \boldsymbol{\nabla}R\right\rangle}
		 {(\mathbf{R} +\beta)^2}
		 -
		 \frac{2 |\mathbf{Ric}|^2|\boldsymbol{\nabla}\mathbf{R}|^2}
		 {(\mathbf{R}+\beta)^3}\\
	&\quad\quad\quad\quad
		-
		\frac{|\mathbf{Ric}|^2 (\del_t-\boldsymbol{\Delta})\mathbf{R}}
		{(\mathbf{R}+\beta)^2}
		+
		\frac{2|\mathbf{Ric}|^2|\mathbf{T}|^2}
		{3(\mathbf{R}+\beta)}
\end{align*}
Now substituting in the evolution equations~\eqref{scalar-heat} and~\eqref{Ricci-heat}, and completing the square on the resulting $|\boldsymbol{\nabla}\mathbf{Ric}|^2$ term, this is equal to
\begin{align*}
	&=
		\int
		-\frac{2\left|
		(\mathbf{R}+\beta)\boldsymbol{\nabla}\mathbf{Ric}
		-\mathbf{Ric}\otimes \boldsymbol{\nabla}R\right|^2}
		{(\mathbf{R}+\beta)^3}
		+
		\frac{2|\mathbf{Ric}|^2|\mathbf{T}|^2}
		{3(\mathbf{R}+\beta)}\\
	&
		+\frac{1}{\mathbf{R}+\beta}
		\Big\{
		4 \mathbf{R}_{cabd}\mathbf{R}^{ab}\mathbf{R}^{dc}
		+ 4 \mathbf{R}^{ab}\boldsymbol{\Delta}\left(
		\mathbf{T}_a^{\phantom{a}c}\mathbf{T}_{cb}
		\right)
		+\frac{2}{3}\mathbf{R}\boldsymbol{\Delta}|\mathbf{T}|^2
		+ 8 \mathbf{R}_{cabd}\mathbf{R}^{ab}
		 \mathbf{T}^{ce}\mathbf{T}_e^{\phantom{e}d}
		 +
		 \frac{4}{3}|\mathbf{Ric}|^2|\mathbf{T}|^2\\
	&\quad\quad\quad\quad\quad\quad\quad
		- 4\mathbf{R}^{ab}(	
		 \boldsymbol{\nabla}_a\boldsymbol{\nabla}^c
		 ( \mathbf{T}_c^{\phantom{c}d}\mathbf{T}_{db} )
		 +
		 \boldsymbol{\nabla}_b\boldsymbol{\nabla}^c
		 ( \mathbf{T}_c^{\phantom{c}d}\mathbf{T}_{da} )
		 )
		 -\frac{2}{3}\mathbf{R}^{ab}
		 \boldsymbol{\nabla}_a\boldsymbol{\nabla}_b |\mathbf{T}|^2
		 \Big\}\\
	&
		-\frac{|\mathbf{Ric}|^2}{(\mathbf{R}+\beta)^2}\Big\{
		\boldsymbol{\Delta}\mathbf{R}
		-
		4\boldsymbol{\nabla}^b\boldsymbol{\nabla}^a
		(\mathbf{T}_a^{\phantom{a}c}\mathbf{T}_{cb})
		+
		2|\mathbf{Ric}|^2
		-
		\frac{2}{3}\mathbf{R}^2
		+
		4 \mathbf{R}^{ab}\mathbf{T}_a^{\phantom{a}c}\mathbf{T}_{cb}
		\Big\}
\end{align*}
Now we gather terms and integrate by parts:
\begin{align*}
&=
	\int
	-\frac{2\left|
	(\mathbf{R}+\beta)\boldsymbol{\nabla}\mathbf{Ric}
	-\mathbf{Ric}\otimes \boldsymbol{\nabla}R\right|^2}
	{(\mathbf{R}+\beta)^3}
	-
	\frac{2|\mathbf{Ric}|^4}{(\mathbf{R}+\beta)^2}
	+
	\frac{2|\mathbf{Ric}|^2|\mathbf{T}|^2}
	{\mathbf{R}+\beta}
	+
	\frac{2|\mathbf{Ric}|^2\mathbf{R}^2}
	{3(\mathbf{R}+\beta)^2}\\
&\quad
	+
	\frac{4\mathbf{R}_{cabd}\mathbf{R}^{ab}\mathbf{R}^{dc}}
	{\mathbf{R}+\beta}
	+
	\frac{8 \mathbf{R}_{cabd}\mathbf{R}^{ab}
		 \mathbf{T}^{ce}\mathbf{T}_e^{\phantom{e}d}}
	{\mathbf{R}+\beta}	 
	-
	\frac{4 |\mathbf{Ric}|^2
	\mathbf{R}^{ab}\mathbf{T}_a^{\phantom{a}c}\mathbf{T}_{cb}}
	{(\mathbf{R}+\beta)^2}\\
&\quad\quad
	-4\boldsymbol{\nabla}^c\left(
	\frac{\mathbf{R}^{ab}}{\mathbf{R}+\beta}  
	\right)
	\boldsymbol{\nabla}_c(\mathbf{T}_{a}^{\phantom{a}d}\mathbf{T}_{db})
	-\frac{2}{3}
	\boldsymbol{\nabla}_c\left( 
		\frac{\mathbf{R}}{\mathbf{R}+\beta}
	\right)
	\boldsymbol{\nabla}^c|\mathbf{T}|^2
	+
	8\boldsymbol{\nabla}_a\left( 
		\frac{\mathbf{R}^{ab}}{\mathbf{R}+\beta}
		\right)
	\boldsymbol{\nabla}^c(\mathbf{T}_{c}^{\phantom{c}d}\mathbf{T}_{db})\\
&\quad\quad\quad
	+\frac{2}{3}\boldsymbol{\nabla}_a\left(  
		\frac{\mathbf{R}^{ab}}{\mathbf{R}+\beta}
	\right)
	\boldsymbol{\nabla}_b|\mathbf{T}|^2
	-4\boldsymbol{\nabla}^b\left( 
	 	\frac{|\mathbf{Ric}|^2}{(\mathbf{R}+\beta)^2}
	 \right)
	 \boldsymbol{\nabla}^a(\mathbf{T}_{a}^{\phantom{a}c}\mathbf{T}_{cb})
\end{align*}
Next we bound this by the norms of various terms, also applying the contracted Bianchi identity, Kato's inequality, and \eqref{torsion-bounded-by-Ric} as appropriate. This means the left-hand side is bounded by
\begin{align*}
&\leq
	\int
	-\frac{2|\mathbf{Ric}|^4}{(\mathbf{R}+\beta)^2}
	+\frac{2|\mathbf{Ric}|^2|\mathbf{T}|^2}{\mathbf{R}+\beta}
	+\frac{2|\mathbf{Ric}|^2\mathbf{R}^2}{3(\mathbf{R}+\beta)^2}
	+\frac{(4+8\sqrt{7})|\mathbf{Rm}||\mathbf{Ric}|^2}
		{\mathbf{R}+\beta}
	+\frac{4|\mathbf{Ric}|^3|\mathbf{T}|^2}{(\mathbf{R}+\beta)^2}\\
&\quad
	\left( 
		\frac{8|\boldsymbol{\nabla}\mathbf{Ric}|}{\mathbf{R}+\beta}
		+
		\frac{8|\mathbf{Ric}||\boldsymbol{\nabla}\mathbf{R}|}
			{(\mathbf{R}+\beta)^2}
	\right)|\mathbf{T}||\boldsymbol{\nabla}\mathbf{T}|
	+\frac{8\beta|\mathbf{T}|^2|\boldsymbol{\nabla}\mathbf{T}|^2}
		{3(\mathbf{R}+\beta)^2}
	+\left( 
	\frac{4|\boldsymbol{\nabla}\mathbf{R}|}{\mathbf{R}+\beta}
	+
	\frac{8|\mathbf{Ric}||\boldsymbol{\nabla}\mathbf{R}|}
		{(\mathbf{R}+\beta)^2}
	\right)|\mathbf{T}||\boldsymbol{\nabla}\mathbf{T}|\\
&\quad
	+\left(  
	\frac{|\boldsymbol{\nabla}\mathbf{R}|}{3(\mathbf{R}+\beta)}
	+
	\frac{2|\mathbf{Ric}||\boldsymbol{\nabla}\mathbf{R}|}
		{3(\mathbf{R}+\beta)^2}
	\right)|\boldsymbol{\nabla}|\mathbf{T}|^2|
	+
	\left(  
	\frac{8|\mathbf{Ric}||\boldsymbol{\nabla}\mathbf{Ric}|}
		{(\mathbf{R}+\beta)^2}
	+
	\frac{8|\mathbf{Ric}|^2|\boldsymbol{\nabla}\mathbf{R}|}
		{(\mathbf{R}+\beta)^3}
	\right)|\mathbf{T}||\boldsymbol{\nabla}\mathbf{T}|\\
&\leq
	\int
	-\frac{2|\mathbf{Ric}|^4}{(\mathbf{R}+\beta)^2}
	+\frac{2|\mathbf{Ric}|^2|\mathbf{T}|^2}{\mathbf{R}+\beta}
	+\frac{2|\mathbf{Ric}|^2\mathbf{R}^2}{3(\mathbf{R}+\beta)^2}
	+\frac{(4+8\sqrt{7})|\mathbf{Rm}||\mathbf{Ric}|^2}
		{\mathbf{R}+\beta}
	+\frac{4|\mathbf{Ric}|^3|\mathbf{T}|^2}{(\mathbf{R}+\beta)^2}\\
&\quad
	+\frac{8|\boldsymbol{\nabla}\mathbf{Ric}|
		|\mathbf{T}||\boldsymbol{\nabla}\mathbf{T}|}
		{\mathbf{R}+\beta}
	+\frac{104|\mathbf{Ric}|
		|\mathbf{T}||\boldsymbol{\nabla}\mathbf{T}|}
		{3(\mathbf{R}+\beta)^2}
	+\frac{8\beta|\mathbf{T}|^2|\boldsymbol{\nabla}\mathbf{T}|^2}
		{3(\mathbf{R}+\beta)^2}
	+\frac{28|\mathbf{T}|^2|\boldsymbol{\nabla}\mathbf{T}|^2}
		{3(\mathbf{R}+\beta)}\\
&\quad
	+\frac{8|\mathbf{Ric}||\boldsymbol{\nabla}\mathbf{Ric}|
		|\mathbf{T}||\boldsymbol{\nabla}\mathbf{T}|}
		{(\mathbf{R}+\beta)^2}
	+\frac{16|\mathbf{Ric}|^2|\mathbf{T}|^2
		|\boldsymbol{\nabla}\mathbf{T}|^2}
		{(\mathbf{R}+\beta)^3}
\end{align*}
We now use Cauchy--Schwarz to isolate the $|\boldsymbol{\nabla}\mathbf{Ric}|$ term. This implies the left-hand side is bounded by
\begin{align*}
&\leq
	\int
	-\frac{2|\mathbf{Ric}|^4}{(\mathbf{R}+\beta)^2}
	+\frac{2|\mathbf{Ric}|^2|\mathbf{T}|^2}{\mathbf{R}+\beta}
	+\frac{2|\mathbf{Ric}|^2\mathbf{R}^2}{3(\mathbf{R}+\beta)^2}
	+\frac{(4+8\sqrt{7})|\mathbf{Rm}||\mathbf{Ric}|^2}
		{\mathbf{R}+\beta}
	+\frac{4|\mathbf{Ric}|^3|\mathbf{T}|^2}{(\mathbf{R}+\beta)^2}\\
&\quad
	+|\boldsymbol{\nabla}\mathbf{Ric}|^2
	+\left(  
	\frac{32}{(\mathbf{R}+\beta)^2}
	+\frac{8\beta}{3(\mathbf{R}+\beta)^2}
	+\frac{28}{3(\mathbf{R}+\beta)}
	\right)|\mathbf{T}|^2|\boldsymbol{\nabla}\mathbf{T}|^2\\
&\quad
	+\left(  
	\frac{104\sqrt{7}}{3(\mathbf{R}+\beta)^2}
	+\frac{32|\mathbf{T}|^2}{(\mathbf{R}+\beta)^4}
	+\frac{16|\mathbf{T}|^2}{(\mathbf{R}+\beta)^3}
	\right)|\mathbf{Ric}|^2|\boldsymbol{\nabla}\mathbf{T}|^2
\end{align*}
Finally, we absorb the positive $|\mathbf{Ric}|^3$ term into the negative $|\mathbf{Ric}|^4$ term using Cauchy--Schwarz, at the expense of introducing a term with $|\Ric|^2$ in it. This gives the bound
\begin{align*}
&\leq
	\int
	-\frac{|\mathbf{Ric}|^4}{(\mathbf{R}+\beta)^2}
	+\frac{(4+8\sqrt{7})|\mathbf{Rm}||\mathbf{Ric}|^2}{\mathbf{R}+\beta}
	+\left( 
	 \frac{2|\mathbf{T}|^2}{\mathbf{R}+\beta}
	 +\frac{14|\mathbf{T}|^4}{3(\mathbf{R}+\beta)^2}
	 \right)|\mathbf{Ric}|^2\\
&\quad
	+\left( 
		\sqrt{7}\left(  
				\frac{32 + 8\beta/3}{(\mathbf{R}+\beta)^2}
				+\frac{28}{3(\mathbf{R}+\beta)}
				\right) 
		+\frac{104\sqrt{7}}{3(\mathbf{R}+\beta)^2}
		+\frac{32|\mathbf{T}|^2}{(\mathbf{R}+\beta)^4}
		+\frac{16|\mathbf{T}|^2}{(\mathbf{R}+\beta)^3}
	\right)|\mathbf{Ric}|^2|\boldsymbol{\nabla}\mathbf{T}|^2\\
&\quad\quad\quad\quad
	+|\boldsymbol{\nabla}\mathbf{Ric}|^2
\end{align*}
A little manipulation, putting coefficients over common denominators, gives the result.
\end{proof}

\begin{corollary}
Let $\underline{\omega}(t)$ be a solution to the hypersymplectic flow on a compact 4-manifold and time interval $t\in [0,s)$. Suppose moreover that $|\mathbf{T}|^2 \leq \beta/2$ for all $t$. Then there exists a constant $C$, depending only on $\beta$, $s$ and the initial data, such that for all $t \in [0,s)$
\[
\frac{\diff}{\diff t} \int \frac{|\mathbf{Ric}|^2}{\mathbf{R}+\beta}
	\leq
		\int - \frac{4}{\beta^2} |\mathbf{Ric}|^4 
			+|\boldsymbol{\nabla}\mathbf{Ric}|^2
			+ C\left( 
			|\mathbf{Rm}||\mathbf{Ric}|^2 
			+ |\mathbf{Ric}|^2|\hat{\nabla}\diff Q|^2_Q
			+ |\mathbf{Ric}|^2 
			\right)
\]
\end{corollary}
\begin{proof}
This follows from \eqref{7D-Ricc4-bound} by substituting the bounds $\mathbf{R}+ \beta \geq \beta/2$, $|\mathbf{T}|\leq C$, and $|\boldsymbol{\nabla}\mathbf{T}|^2 \leq C (|\hat{\nabla}\diff Q|^2_Q+1)$.
\end{proof}

\subsection{Evolution of integral quantities in 4-dimensions}

We move now to the evolution of $\int |\diff Q|^2_Q$. Proposition~\ref{dQ-mp} shows that there is a constant $C$ (depending only on $\tr Q$) such that
\begin{equation}\label{dQ-heat}
(\del_t - \Delta) |\diff Q|^2_Q
\leq
-|\hat{\nabla}\diff Q|^2_Q + C|\mathbf{T}|^2|\diff Q|^2_Q
\end{equation}

\begin{lemma}
Let $\underline{\omega}(t)$ be a solution to the hypersymplectic flow and suppose that $\tr Q < c$ is bounded for all $t \in [0,s)$. Then there is a constant $C$ depending only on $c$ such that
\[
\frac{\diff} {\diff t} \int |\diff Q|^2_Q 
	\leq
	\int \left(- |\hat{\nabla}\diff Q|^2_Q + C |\mathbf{T}|^2|\diff Q|^2_Q\right)
\]
In particular, this gives the fourth inequality of Proposition~\ref{differential-inequalities}.
\end{lemma}
\begin{proof}
Directly from \eqref{dQ-heat}, we have
\[
\frac{\diff}{\diff t} \int|\diff Q|^2_Q
	=
		\int\left( \del_t - \Delta \right)|\diff Q|^2_Q 
		+ 
		\frac{2}{3} |\diff Q|^2_Q|\mathbf{T}|^2
	\leq
		\int  -|\hat{\nabla}\diff Q|^2 
		+ 
		\left(C + \frac{2}{3}\right)|\diff Q|^2_Q
\qedhere\]
\end{proof}

It remains to prove the third inequality of Proposition~\ref{differential-inequalities}:

\begin{lemma}
Let $\underline{\omega}(t)$ be a solution to the hypersymplectic flow on a compact 4-manifold and time interval $t\in [0,s)$. Suppose moreover that $|\mathbf{T}|^2 \leq \beta/2$ for all $t$. Then there exists a constant $C$, depending only on $\beta$, $s$ and the initial data, such that for all $t \in [0,s)$
\[
\frac{\diff}{\diff t} \int |\mathbf{Ric}|^2|\diff Q|^2_Q
	\leq
		\int
		-\frac{1}{2}|\mathbf{Ric}|^2|\hat{\nabla}\diff Q|^2_Q
		+C\left(  
		|\mathbf{Rm}||\mathbf{Ric}|^2 
		+|\hat{\nabla}\diff Q|^2_Q
		+|\boldsymbol{\nabla}\mathbf{Ric}|^2
		+|\mathbf{Ric}|^2
		+ 1
		\right)
\]
\end{lemma}
\begin{proof}
We being the computation by substituting \eqref{Ricci-heat} and \eqref{dQ-heat}:
\begin{align*}
\frac{\diff}{\diff t}\int |\mathbf{Ric}|^2|\diff Q|^2_Q
	&=
		\int 
		|\diff Q|^2_Q(\del_t -\Delta)|\mathbf{Ric}|^2
		+|\mathbf{Ric}|^2(\del_t - \Delta)|\diff Q|^2_Q\\
	&\quad\quad
		-2\left\langle \boldsymbol{\nabla}|\mathbf{Ric}|^2,
		\boldsymbol{\nabla}|\diff Q|^2_Q \right\rangle
		+\frac{2}{3}|\mathbf{Ric}|^2|\diff Q|^2_Q|\mathbf{T}|^2\\
	&\leq
		\int
		|\diff Q|^2_Q\Big\{
			- 2 |\boldsymbol{\nabla}\mathbf{Ric}|^2
			+ 4 \mathbf{R}_{cabd}\mathbf{R}^{ab}\mathbf{R}^{dc}
			+ 4 \mathbf{R}^{ab}\boldsymbol{\Delta}\left(
			\mathbf{T}_a^{\phantom{a}c}\mathbf{T}_{cb}
			\right)\\
	&\quad\quad\quad\quad\quad\quad
			+\frac{2}{3}\mathbf{R}\boldsymbol{\Delta}|\mathbf{T}|^2	
			- 4\mathbf{R}^{ab}\left(	
			 \boldsymbol{\nabla}_a\boldsymbol{\nabla}^c
			 \left( \mathbf{T}_c^{\phantom{c}d}\mathbf{T}_{db} \right)
			 +
			 \boldsymbol{\nabla}_b\boldsymbol{\nabla}^c
			 \left( \mathbf{T}_c^{\phantom{c}d}\mathbf{T}_{da} \right)
			 \right)\\
	&\quad\quad\quad\quad\quad\quad
			 -\frac{2}{3}\mathbf{R}^{ab}
			 \boldsymbol{\nabla}_a\boldsymbol{\nabla}_b |\mathbf{T}|^2
			 + 8 \mathbf{R}_{cabd}\mathbf{R}^{ab}
			 \mathbf{T}^{ce}\mathbf{T}_e^{\phantom{e}d}
			 +
			 \frac{4}{3}|\mathbf{Ric}|^2|\mathbf{T}|^2
			 \Big\}\\
	&\quad\quad\quad\quad\quad\quad\quad\quad\quad\quad
		+|\mathbf{Ric}|^2
		(-|\hat{\nabla}\diff Q|^2_Q + C|\mathbf{T}|^2|\diff Q|^2_Q)\\
	&\quad\quad\quad\quad\quad\quad\quad\quad\quad\quad
		+ 8|\diff Q|_Q|\boldsymbol{\nabla}\mathbf{Ric}|
			|\mathbf{Ric}||\hat{\nabla}\diff Q|_Q
		+\frac{2}{3}|\mathbf{Ric}|^2|\diff Q|^2_Q|\mathbf{T}|^2
\end{align*}	
Then we integrate by parts to bound this by
\begin{align*}
&\leq
	\int
	(4+8\sqrt{7})|\diff Q|^2_Q|\mathbf{Rm}||\mathbf{Ric}|^2
	+(2+C)|\mathbf{Ric}|^2|\diff Q|^2_Q|\mathbf{T}|^2\\
&\quad\quad
	+8\Big(  
		|\boldsymbol{\nabla}|\diff Q|^2_Q| |\mathbf{Ric}|
			+
		|\diff Q|^2_Q |\boldsymbol{\nabla}\mathbf{Ric}|
	\Big) |\mathbf{T}||\boldsymbol{\nabla}\mathbf{T}|
	+\frac{2}{3}\Big(
		|\boldsymbol{\nabla}|\diff Q|^2_Q| |\mathbf{R}|
			+
		|\diff Q|^2_Q |\boldsymbol{\nabla}\mathbf{R}|
	\Big)|\boldsymbol{\nabla}|\mathbf{T}|^2|\\
&\quad\quad
	+8\Big(  
		|\boldsymbol{\nabla}|\diff Q|^2_Q| |\mathbf{Ric}|
			+
		\frac{1}{2}|\diff Q|^2_Q |\boldsymbol{\nabla}\mathbf{R}|
	\Big) |\mathbf{T}||\boldsymbol{\nabla}\mathbf{T}|	
	+\frac{2}{3}\Big(
		|\boldsymbol{\nabla}|\diff Q|^2_Q| |\mathbf{Ric}|
			+
		\frac{1}{2}|\diff Q|^2_Q |\boldsymbol{\nabla}\mathbf{R}|
	\Big)|\boldsymbol{\nabla}|\mathbf{T}|^2|\\
&\quad\quad\quad\quad\quad
	+8 |\diff Q|_Q|\boldsymbol{\nabla}\mathbf{Ric}|
		|\mathbf{Ric}| |\hat{\nabla}\diff Q|_Q
	-|\mathbf{Ric}|^2 |\hat{\nabla}\diff Q|^2_Q
	-2|\diff Q|^2_Q|\boldsymbol{\nabla}\mathbf{Ric}|^2
\end{align*}
\begin{align*}
&\leq
	\int
	(4+8\sqrt{7})|\diff Q|^2_Q|\mathbf{Rm}||\mathbf{Ric}|^2
	+(2+C)|\mathbf{Ric}|^2|\diff Q|^2_Q|\mathbf{T}|^2\\
&\quad\quad
	+\frac{104}{3} |\diff Q|_Q|\mathbf{T}||\hat{\nabla}\diff Q|_Q
		|\mathbf{Ric}||\boldsymbol{\nabla}\mathbf{T}|
	+\frac{8}{3}|\diff Q|_Q |\mathbf{T}|^3|\hat{\nabla}\diff Q|_Q
		|\boldsymbol{\nabla}\mathbf{T}|\\
&\quad\quad
	+12|\diff Q|^2_Q|\mathbf{T}|^2|\boldsymbol{\nabla}\mathbf{T}|^2
	+8|\diff Q|^2_Q|\boldsymbol{\nabla}\mathbf{Ric}|
			|\mathbf{T}||\boldsymbol{\nabla}\mathbf{T}|\\
&\quad\quad\quad\quad\quad
	+8 |\diff Q|_Q|\boldsymbol{\nabla}\mathbf{Ric}|
		|\mathbf{Ric}| |\hat{\nabla}\diff Q|_Q
	-|\mathbf{Ric}|^2 |\hat{\nabla}\diff Q|^2_Q
	-2|\diff Q|^2_Q|\boldsymbol{\nabla}\mathbf{Ric}|^2
\end{align*}
Next we apply Cauchy--Schwarz in various places, giving
\begin{align*}
&\leq
	\int
	(4+8\sqrt{7})|\diff Q|^2_Q|\mathbf{Rm}||\mathbf{Ric}|^2
	+(2+C)|\mathbf{Ric}|^2|\diff Q|^2_Q|\mathbf{T}|^2\\
&\quad\quad
	+\left(  
		\frac{1}{4} |\mathbf{Ric}|^2|\hat{\nabla}\diff Q|^2_Q
		+
		\left( \frac{104}{3} \right)^2
			|\mathbf{T}|^2|\diff Q|^2_Q|\boldsymbol{\nabla}\mathbf{T}|^2
	\right)
	+\left(  
		\frac{4}{3}|\mathbf{T}|^4|\hat{\nabla}\diff Q|^2_Q
		+
		\frac{4}{3}|\diff Q|^2_Q|\mathbf{T}|^2
			|\boldsymbol{\nabla}\mathbf{T}|^2
	\right)\\
&\quad\quad
	+12|\diff Q|^2_Q|\mathbf{T}|^2|\boldsymbol{\nabla}\mathbf{T}|^2
	+\left(  
		2|\diff Q|^2_Q|\boldsymbol{\nabla}\mathbf{Ric}|^2
		+
		8|\diff Q|^2_Q|\mathbf{T}|^2|\boldsymbol{\nabla}\mathbf{T}|^2
	\right)\\
&\quad\quad
	+\left(  
	\frac{1}{4}|\mathbf{Ric}|^2|\hat{\nabla}\diff Q|^2_Q
	+
	64|\diff Q|^2_Q|\boldsymbol{\nabla}\mathbf{Ric}|^2
	\right)
	-|\mathbf{Ric}|^2 |\hat{\nabla}\diff Q|^2_Q
	-2|\diff Q|^2_Q|\boldsymbol{\nabla}\mathbf{Ric}|^2
\end{align*}
Collecting the pieces, we see we have isolated the negative term $-\frac{1}{2}|\mathbf{Ric}|^2 |\hat{\nabla}\diff Q|^2_Q$. The result now follows from the bounds $|\diff Q|_Q < C$ and $|\boldsymbol{\nabla}\mathbf{T}|^2 \leq C (|\hat{\nabla}\diff Q|^2_Q+1)$ which follow from the assumed bound on~$|\mathbf{T}|$. 
\end{proof}
This Lemma completes the proof of Proposition~\ref{differential-inequalities} and hence the proof of Theorem~\ref{Ricci-L2-bound}.

\section{The hypersymplectic flow extends as long as torsion is bounded}
\label{proof-of-main-result}

We now turn to the proof of our main result, Theorem~\ref{main-theorem}. A key ingredient is the recent result of Gao Chen \cite[Thm. 1.1]{Chen}, generalising Perelman's noncollapsing theorem for Ricci flow \cite{perelman} (see also the discussion in \cite[section 13]{KL})  We begin by recalling the notion of $\kappa$-noncollapsing introduced by Perelman.  

\begin{definition} 
An $n$-dimensional Riemmanian manifold $(M,g)$ is said to be \emph{$\kappa$-noncollapsed relative to an upper bound of scalar curvature on the scale $\rho$} if for all $B_g(x,r) \subset M$ with $r < \rho$, 
\[
\text{if }\sup_{B(x,r)} R \leq \frac{1}{r^2}\quad
\text{then }
\vol_g\left( B_{g}(x,r) \right) \geq \kappa r^n
\]
\end{definition}

We now consider an arbitrary path of Riemannian metrics $g(t)$ and write 
\begin{equation}
\del_t g = -2 \Ric + P
\label{general-flow}
\end{equation}
where $P$ indicates the difference between $g(t)$ and a solution of Ricci flow.

\begin{theorem}[Chen \cite{Chen}, cf.\ Perelman \cite{KL}]\label{noncollapse}
If $|P|$ is bounded along the flow \eqref{general-flow} for $t\in [0,s)$ with $s<\infty$, then there exists $\kappa>0$ such that for all $t \in [0,s)$, $g(t)$ is $\kappa$-noncollapsed relative to an upper bound of scalar curvature on the scale $\rho = \sqrt{s}$.
\end{theorem} 
 
With this in hand, we now give the proof of Theorem~\ref{main-theorem}. We assume that $\underline{\omega}(t)$ is a solution of the hypersymplectic flow for $t \in [0,s)$, with $s<\infty$, that $|\mathbf{T}| < C$ for all $t\in [0,s)$ and that, for a contradiction, the flow does not extend to $t=s$. 

Write $g(t) = g_{\underline{\omega}(t)}$. By Propositions~\ref{trQ-mp} and~\ref{dQ-mp}, $\tr Q$ and $|\diff Q|_Q$ are bounded along the flow, and hence $|\underline{\tau}|$ is also bounded. From this and the evolution equation for $g_{\underline{\omega}}$ we see that the ``error term''
\[
P 
	:= 
		\del_t g + 2\Ric
	= 
		\frac{1}{2} \left\langle \diff Q,\diff Q \right\rangle_Q 
		+ \tr (Q^{-1}\tau \otimes \tau)
		-\frac{2}{3}|\mathbf{T}|^2 g
\]
is  bounded along the flow. Moreover, by Lemma~\ref{Ricci-splitting}, $R = \frac{1}{4}|\diff Q|^2_Q - |\mathbf{T}|^2$ and so $R$ is also uniformly bounded.  Theorem~\ref{noncollapse} then implies that there exists $\kappa>0$ and $\rho>0$ such that for any $p \in X$,  $t \in [0,s)$ and $r \leq \rho$,
\begin{equation}
\vol_{g(t)}B_{g(t)}(p,r) \geq \kappa r^4
\label{vol-lower-bound}
\end{equation}

Write
\[
\Lambda(x,t)
	=
		\left( |\mathbf{Rm}|^2(x,t) 
		+ |\boldsymbol{\nabla}\mathbf{T}|^2(x,t) 
		\right)^{1/2}
\]
Since we have assumed that the flow does not extend to $t=s$, it follows from Lotay--Wei's extension result \cite[Theorem~1.3]{Lotay--Wei} that there is a sequence $(p_k,t_k) \in M \times [0,s)$ such that $t_k \to s$ and $\Lambda(p_k, t_k) \to \infty$. Moreover, we can choose $(p_k,t_k)$ so that 
\[
\Lambda(p_k,t_k) = \sup \{ \Lambda(x,t) : x\in M, t\in [0,t_k]\}
\]

To lighten the notation we write $\Lambda_k:=\Lambda(p_k,t_k)$. We define a sequence of flows by parabolic rescaling; for $t\in [-\Lambda_t t_k, 0]$ we set
\begin{align*}
\phi^{(k)}(t)
	&=
		\Lambda_k \phi \left( 
			\Lambda_k^{-1} t + t_k 
		\right)\\
\underline{\omega}^{(k)}(t)
	&=
		\Lambda_k \underline{\omega}\left( 
			\Lambda_k^{-1} t + t_k 
		\right)\\
\boldsymbol{g}^{(k)}(t)
	&=
		\Lambda_k \boldsymbol{g}\left( 
			\Lambda_k^{-1} t + t_k 
		\right)		
\end{align*}
Here $\phi^{(k)}(t)$ is a sequence of $G_2$-Laplacian flows corresponding to the sequence $\underline{\omega}^{(k)}(t)$ of hypersymplectic flows, with induced metrics $\boldsymbol{g}^{(k)}(t)$ on $X \times \mathbb{T}^3$. Write $\Lambda^{(k)}(x,t)$ for the quantity analogous to $\Lambda(x,t)$, but defined using $\phi^{(k)}(t)$. By construction, 
\[
\Lambda^{(k)}(x,t)
	=
\Lambda_k^{-1} \Lambda(x,\Lambda_k^{-1}t + t_k)
	\leq 1
\]
for any $(x,t) \in X \times [-\Lambda_k t_k, 0]$. Now the Bando--Bernstein--Shi  estimates proved by Lotay--Wei \cite{Lotay--Wei} tell us that for any $A>0$, $l \in \N$, there exists $c_{l,A}$ such that
\[
\sup_{X \times [-A,0]} 
\left( |\boldsymbol{\nabla}^{l+1} \mathbf{T}^{(k)}|
+
|\boldsymbol{\nabla}^l \mathbf{Rm}(\boldsymbol{g}^{(k)})| \right)
\leq c_{l,A}
\]
It follows from Corollary~\ref{all-derivs-Rm-bound} and Lemmas~\ref{nabla-omega-bound} and~\ref{nabla2-omega-bound}, and Lemma~\ref{HessianQ-bound} that $|\nabla \omega_i^{(k)}|$, $|\nabla^2\omega_i^{(k)}|$ and $|\nabla^l\Rm(g^{(k)})|$ are uniformly bounded for any choice of $l$.

By Cheeger--Gromov--Taylor \cite[Theorem~4.7]{CGT}, the volume lower bound \eqref{vol-lower-bound} and uniform curvature bound imply that the injectivity radius of $g^{(k)}(0)$ is uniformly bounded below away from zero. It then follows from Cheeger--Gromov compactness that we can take a limit: there exists a pointed Riemmanian 4-manifold $(X^\infty, g^\infty, p_\infty)$ such that
\[
(X, g^{(k)}(0),p_k) 
	\xrightarrow{\text{Cheeger--Gromov}} 
		(X^\infty,g^{(\infty)}(0),p_\infty)
\]
In other words, there is an exhaustion $X^\infty = \bigcup \Omega_k$ by nested neighbourhoods of $p_\infty$ and diffeomorphic injections $f_k \colon \Omega_k \to X$ with $f_k(p_\infty) = p_k$ such that on any fixed compact subset $K \subset X^\infty$, we have $f_k^* g^{(k)}(0) \to g^{\infty}$.

It follows from Proposition~\ref{trQ-mp} and Lemmas~\ref{HessianQ-bound}, \ref{nabla-omega-bound} and~\ref{nabla2-omega-bound} that we have a uniform $C^2$-bound on $\underline{\omega}$. By Arzela--Ascoli, we pass to a subsequence for which $f_k^*\underline{\omega}^{(k)}(0)$ converges in $C^{1,\alpha}$ to a limiting triple of closed self-dual 2-forms $\underline{\omega}^{\infty}$. Write $Q^\infty_{ij} = \frac{1}{2}\left\langle \underline{\omega}^\infty_i , \underline{\omega}^\infty_j\right\rangle$ for the matrix of pointwise innerproducts. Now $Q^\infty$ is the limit (under the diffeomorphisms $f_k$) of the sequence $Q^{(k)}(t)$ of matrices defined by $\underline{\omega}^{(k)}(t)$. But $Q$ is a scale-invariant quantity and so $Q^{(k)}(t) = Q(t)$, the lowest eigenvalue of which is uniformly bounded below away from zero. It follows that $Q^\infty$ is positive definite and so $\underline{\omega}^{\infty}$ is a hypersymplectic structure determining the metric $g_\infty$. 

We now compute 
\[
|\diff Q^{\infty}|^2_{Q^\infty}
	=
		\lim_{k\to \infty} |\diff Q^{(k)}|^2_{Q^{(k)}}(0)
	=
		\lim_{k\to \infty} \Lambda^{-1}_k |\diff Q|^2_{Q}(t_k)
	=
		0
\]
where the middle norm is taken with respect to $g^{(k)}(0)$, the right-hand norm with $g(t_k)$ and we have used the facts that $Q^{(k)} = Q$ and that $|\diff Q|^2_{Q}$ is uniformly bounded. It follows that $g^\infty$ is hyperk\"ahler and $\underline{\omega}^\infty$ is a constant rotation of a hyperk\"ahler triple of 2-forms.

The uniform bound of the scale invariant energy $\int |\Rm|^2$ along the flow (Corollary~\ref{4D-energy-bound}) shows that
\[
\int_{X^\infty} |\Rm(g^{\infty})|^2 < \infty
\]
which means that $(X^\infty, g^\infty)$ is a gravitational instanton. Moreover, since 
\[
|\Rm(g^\infty)|(p_\infty,0) 
	= \Lambda_{g^\infty} 
	=\lim_{k \to \infty} \Lambda^{-1}_k \Lambda_{g(t_k)}(p_k, t_k)
	= 1
\]
we see that $(X^\infty, g^\infty)$ is a \emph{non-trivial} gravitational instanton. 

The volume ratio lower bound \eqref{vol-lower-bound} implies $(X^\infty, g^\infty)$ has Euclidean volume growth and so, by a theorem of Bando--Kasue--Nakajima \cite{BKN}, $(X^\infty, g^\infty)$ is asymptotically locally Euclidean. We now invoke Kronheimer's classification of asymptotically locally Euclidean gravitational instantons \cite{Kron}, which implies that there is at least one 2-sphere $S \subset X^\infty$ which is holomorphic for one of the hyperk\"ahler complex structures $J$. Without loss of generality we assume this $J$ corresponds to the K\"ahler form $\omega^\infty_1$ (otherwise we apply a linear transformation, constant in space and time, to the triples). We also note that such a 2-sphere automatically has $[S]^2 = -2$. (This follows from the adjunction formula and the fact that the canonical bundle of a hyperk\"ahler manifold is trivial.)

Recall the diffeomorphisms $f_k \colon \Omega_k \to X$ in the Cheeger--Gromov convergence. For any $i=1,2,3$, we have
\[
\int_S \omega^\infty_i 
=
\lim_{k\to\infty} \int_S f^*_k \omega^{(k)}_i(0) 
=
\lim_{k\to \infty} \int_{f_k(S)} \omega^{(k)}_i(0)
=
\lim_{k\to \infty }\Lambda_k \left\langle [\omega_i], [f_k(S)] \right\rangle
\]
It follows that $\left\langle [\omega_i], [f_k(S)] \right\rangle \to 0$. Meanwhile, since $S$ is a symplectic submanifold for $\omega_1^\infty$, it follows that it is symplectic for $f_k^*\omega_1^{(k)}(0)$ for all large $k$. For these values of $k$ we have $\left\langle [\omega_1], [f_k(S)] \right\rangle >0$.

We will now deduce a contradiction. Note that $b_+=3$ and the $[\omega_i]$'s span a maximal positive-definite subspace of $H^2(X,\R)$ for the cup product (as is explained in \S\ref{intro}, straight after Conjecture~\ref{skd-conjecture}). Set
\[
H^-_2 = \{ Z \in H_2(X,\R) : \left\langle [\omega_i],Z \right\rangle=0 
\text{ for }i=1,2,3\}
\]
By Poincar\'e duality this is a maximal negative-definite subspace of $H_2(X,\R)$ for the intersection form. Its orthogonal complement $H^+_2\subset H^2(X,\R)$ (defined via the intersection form) is a maximal positive-definite subspace. Now write $[f_k(S)] = P_k + N_k$ with $P_k \in H^+_2$ and $N_k \in H^-_2$. Since $\left\langle [\omega_i], [f_k(S)] \right\rangle \to 0$ we see that $P_k \to 0$. Since $[f_k(S)]^2=-2$ we deduce that $N_k^2 \to -2$ and so $N_k$ is bounded. It follows that $[f_k(S)]$ lies in a bounded set in $H_2(X,\R)$, but it also lies in the integral lattice $H_2(X,\Z)$ and hence it can take at most \emph{finitely} many different values. Now $[\omega_1]$ has a smallest non-zero value on this finite set, $v>0$ say. It follows that for all large $k$, $\left\langle [\omega_1], [f_k(S)] \right\rangle \geq v >0$ which contradicts the fact that this same sequence converges to zero. This contradiction completes the proof of Theorem~\ref{main-theorem}.\\

We finish the article with an application of Theorem~\ref{main-theorem}, giving an interesting estimate for the maximal existence time of a hypersymplectic flow, in terms of some quite weak bounds on the initial data (compared, for example, with the  ``doubling-time estimate'' of \cite[Proposition~4.1]{Lotay--Wei}). 

\begin{theorem}\label{lower-bound-existence-time}
There exists a constant $C>0$ such that any hypersymplectic flow $\underline{\omega}(t)$ exists on the interval $[0,s)$ for 
\[
s \geq \frac{1}{2CK^{21}A}
\]
where $K = \sup \tr Q(0)$ and $A= \sup|\diff Q|_Q^2(0)$. 
\end{theorem}

\begin{proof}
By Propositions~\ref{trQ-mp} and~\ref{dQ-mp} there is a constant $C$ such that
\begin{align*}
(\del_t - \Delta)|\diff Q|^2_Q 
	&\leq 
		C (\tr Q) ^{21}|\mathbf{T}|^2|\diff Q|^2_Q\\
(\del_t - \Delta)\tr Q
	& \leq 
		\frac{5}{3}|\mathbf{T}|^2 \tr Q
\end{align*}
Let
\begin{align*}
f(t) & = \sup |\diff Q|^2_Q(t)\\
f_1(t)& = \sup \tr Q(t)\\
f_2(t) &= \sup |\mathbf{T}|^2(t)
\end{align*}
The heat inequalities and the maximum principle give
\begin{align*}
f'
	&\leq
		C f_1^{21}f_2 f\\
(f_1^{21})' 
	&
		\leq 35 f_2f_1^{21}
\end{align*}
from which it follows that
\[
(f_1^{21}f)' = (f_1^{21})' f + f_1^{21}f'
\leq 35 f_1^{21}f_2f + Cf_1^{42}f_2f \leq C'(f_1^{21}f)^2
\]
The consequence is that
\[
f_1^{21}(t)f(t) \leq
\frac{1}{(f_1^{21}f)^{-1}(0) - C't}
\]

Now, assume for a contradiction that the maximal time of the flow is $s\leq \frac{1}{2C'K^{21}A}$. Then the above bound shows that
\[
\sup_{[0,s)} f(t) < \infty
\]
and hence $|\mathbf{T}|^2 \leq \frac{3}{2}|\diff Q|^2_Q$ is uniformly bounded on $[0,s)$. But, by Theorem~\ref{main-theorem}, this contradicts the fact that $s$ is maximal.
\end{proof}

\bigskip
{ \footnotesize
\noindent\textsc{J.~Fine and C.~Yao},\\ D\'epartement de math\'ematiques, Universit\'e libre de Bruxelles, Brussels 1050, Belgium.\\ \textit{E-mail addresses:} \texttt{\href{mailto:joel.fine@ulb.ac.be}{joel.fine@ulb.ac.be}}, \texttt{\href{mailto:chengjian.yao@gmail.com}{chengjian.yao@gmail.com}}
}

\end{document}